\documentclass[11pt]{amsart}

\usepackage{mathrsfs,graphicx,latexsym,tikz,color,euscript}
\usepackage{amsfonts,amssymb,amsmath,amscd,amsthm}
\usepackage{epstopdf}
\usepackage{hyperref}
\usepackage{upgreek}
\usepackage{comment} 
\usepackage{pdfsync}

\usepackage{marginnote}
\newcommand{\rr}{\mathbb R}

\newcommand{\M}{\mathcal M}

\newcommand{\ep}{\varepsilon}
\newcommand{\tht}{\theta}
\newcommand{\om}{\varpi}

\newcommand{\ey}{\frac{1}{2}}
\newcommand{\fe}{\mathfrak{e}}

\newtheorem{thm}{Theorem}[section]
\newtheorem{lemma}{Lemma}[section]
\newtheorem{prop}{Proposition}[section]
\newtheorem{remark}{Remark}[section]
\newtheorem{defi}{Definition}[section]
\newtheorem{coro}{Corollary}[section]

\title[relative periodic solution in perturbed Kepler problem]{Relative periodic solutions in spatial Kepler problem with symmetric perturbation}

\author{Xijun Hu}
\address{School of Mathematics, Shandong University, Jinan, China}
\address{State Key Laboratory of Cryptography and Digital Economy Security, Shandong University, Jinan, China}
\email{xjhu@sdu.edu.cn}

\author{Zhiwen Qiao}
\address{School of Mathematics, Shandong University, Jinan, China}
\email{qiaozw@mail.sdu.edu.cn}

\author{Guowei Yu*}
\address{Chern Institute of Mathematics and LPMC, Nankai University, Tianjin, China}
\email{yugw@nankai.edu.cn}

\thanks{This work is supported by the National Key R\&D Program of China (2020YFA0713303). The first author is also supported by the National Natural Science Foundation of China (No. 12521001) and Taishan Scholars Climbing Program of Shandong (TSPD20240802). The last author is also supported by NSFC (No. 12171253), Nankai Zhide Fundation and the Fundamental Research Funds for the Central Universities.}

\thanks{* Corresponding author: Guowei Yu}

\begin{document}

\begin{abstract}

The spatial Kepler problem with a perturbation satisfying the rotational symmetry w.r.t. the $z$-axis and the reflection symmetry w.r.t. the $(x, y)$-plane, can be reduced to an Hamiltonian system with 2 degrees of freedom after fixing the angular momentum. For small enough perturbations, we show that for certain choices of energy and angular momentum, the corresponding energy surface is compact and diffeomorphic to $\mathbb{S}^3$, and on each compact energy surface there is a unique $z$-symmetric brake orbit, which forms a Hopf link with a planar relative periodic orbit. Moreover under some additional technical assumptions, by applying recent results from symplectic dynamics (Cristofaro-Gardiner et al 2023 Geom. Topol. 27 3801–31) and Franks' Theorem, we prove there are infinitely many relative periodic orbits on each compact energy surface. These results can be applied to the motion of a satellite around a uniformly mass-distributed ellipsoid and the $n$-pyramidal problem, where one point mass moves along the $z$-axis and $n$ other equal point masses form a regular $n$-gon perpendicular to the $z$-axis.

\end{abstract}

\maketitle

\section{Introduction}\label{sec:Introduction}

In this paper, we consider the following perturbed spatial Kepler problem
\begin{equation}
	\label{eq;Kepler-perturb} \ddot{q} = -\frac{q}{|q|^3} -\ep \nabla_{q} f(q, \ep), \; q \in \rr^3,
\end{equation}
with the perturbation $f \in C^{\infty}(\rr^3\backslash\{(0,0,0)\} \times [0, 1), \rr)$ satisfying the following symmetry conditions,
\begin{equation}
	\label{eq;symm-cond} \begin{cases}
		\partial_{\tht} f(r, \tht, z, \ep) = 0, \; & \text{rotational symmetry w.r.t. the $z$-axis}, \\
		f(r, \tht, z, \ep) = f(r, \tht, -z, \ep), \; & \text{reflection symmetry w.r.t. the $(x,y)$-plane},\\
	\end{cases}
\end{equation}
where $(r, \tht, z)$ are the cylindrical coordinates of $\rr^3$.


In celestial mechanics, one usually assumes that the celestial bodies are point masses without any size. Such an assumption is justified by the fact that when a celestial body is a perfect sphere with uniform mass distribution, outside the region of such a sphere, its gravitational force is equivalent to a point mass with all the mass concentrated at the center of the sphere. 

Meanwhile a planet (or a star) is usually not a perfect sphere, but closer to an ellipsoid with a small eccentricity. In this case, the gravitational force can be seen as a Kepler problem with a symmetric perturbation satisfying conditions \eqref{eq;symm-cond}. To be more precise, let's assume the center of the ellipsoid is located at the origin, and in spherical coordinates, 
$$  x=\rho\cos\theta\cos\varphi, \; y=\rho\sin\theta\cos\varphi, \; z=\rho\sin\varphi, $$
its boundary is a surface of revolution with respect to the $z$-axis satisfying 
\begin{equation}
	\rho=R\left( 1-\frac{2}{3}\fe P_{2}(\sin\varphi)\right), \; \text{ where } P_{2}(x) = \frac{3x^{2}-1}{2}.
\end{equation}
Here $R$ denotes the mean radius, $\fe$ the eccentricity, and $P_2$ the second Legendre polynomial. For $\fe>0$, the spheroid is oblate (flattened along the $z$-axis), while for $\fe<0$, it is prolate (elongated along the $z$-axis), for details see \cite[page 29-33]{Fit12}.

For an axially symmetric ellipsoid with total mass $M$ and gravitational constant $G$, the kinetic and potential energies of a massless satellite are given by
\begin{equation} \label{eq;ell-U}
	\begin{aligned}
		K_\om(\dot{\rho},\dot{\varphi}) &= \frac{1}{2}\left(\dot{\rho}^2 + \rho^2\dot{\varphi}^2\right) + \frac{\om^{2}}{2\rho^{2}\cos^{2}\varphi}, \\			  
		U_\fe(\rho,\varphi) &= -\frac{GM}{\rho} + \frac{2}{5}\fe\frac{GMR^{2}}{\rho^{3}}P_2(\sin\varphi) + O(\fe^{2}),
	\end{aligned}
\end{equation}
where $\om = \rho^2\cos^2\varphi\,\dot{\theta}$ represents the conserved angular momentum about the $z$-axis.

Setting $GM=1$ and working in cylindrical coordinates, we observe that the potential function constitutes a small perturbation of the Kepler problem satisfying symmetry conditions \eqref{eq;symm-cond} when $|\fe|$ is small. We hereafter refer to this system as \textbf{the ellipsoid problem}. For fixed energy $h<0$ and non-zero angular momentum $\om$ about the $z$-axis, the energy surface
\begin{equation*}
	\mathcal{M}(h,\om,\fe) = \left\{(\dot{\rho},\dot{\varphi},\rho,\varphi): K_\om(\dot{\rho},\dot{\varphi}) + U_\fe(\rho,\varphi) = h\right\}
\end{equation*}
forms a compact manifold diffeomorphic to $\mathbb{S}^{3}$. An orbit is called a \textit{relative periodic solution} if it becomes periodic modulo the $\theta$-coordinate. Our main result states:
\begin{thm}\label{thm:axial symmetry ellipsoid potential}
	The ellipsoid problem admits infinitely many relative periodic solutions within the compact energy surface $\mathcal{M}(h,\om,\fe)$ for sufficiently small eccentricity $\fe$.
\end{thm}

This result follows as a corollary of our analysis for the perturbed Kepler system \eqref{eq;Kepler-perturb}. To elaborate, we introduce momentum variables associated with cylindrical coordinates $(r, \theta, z)$:
\begin{equation*}
	p_r = \dot{r}, \quad p_{\theta} = r^2 \dot{\theta}, \quad p_z = \dot{z}.
\end{equation*}
The corresponding Hamiltonian becomes
\begin{equation}
	\label{eq;Ham} 
	H_{\varepsilon}(p_r, p_{\theta}, p_z, r, \theta, z) = \frac{1}{2}\left( p_{r}^2 + r^{-2} p_{\theta}^2 + p_z^2 \right) - \frac{1}{\sqrt{r^2 + z^2}} + \varepsilon f(r, z, \varepsilon).
\end{equation}
Here $p_{\theta} = r^2\dot{\theta}$ represents the angular momentum about the $z$-axis. By fixing this momentum to a nonzero constant $\om$, we obtain the reduced Hamiltonian
\begin{equation}
	\label{eq:spacial kepler pertur Ham}
	H_{\om, \varepsilon}(p_{r},p_{z},r,z) = \frac{1}{2}(p_{r}^{2} + p_{z}^{2}) + \frac{\om^{2}}{2r^{2}} - \frac{1}{\sqrt{r^{2} + z^{2}}} + \varepsilon f(r,z,\varepsilon).
\end{equation}
The corresponding Hamiltonian equation is 
\begin{equation}
	\label{eq;Ham-sys-pert} (\dot{p}_r, \dot{p}_z, \dot{r}, \dot{z})^T = J \nabla H_{\om, \ep}(p_r, p_z, r, z),
\end{equation} where $J=\begin{pmatrix}
0 & -I_2\\
I_2 & 0
\end{pmatrix}$ is the standard symplectic matrix.

We emphasize that solutions of this reduced system correspond to solutions of the original system with fixed angular momentum $\om$. In particular, periodic solutions in the reduced system generate relative periodic solutions in the full system through rotational symmetry.

The reduced Hamiltonian system possesses two degrees of freedom, making its energy surface
\begin{equation*}
	\mathcal{M}(h, \om, \ep) = \left\{ (p_r, p_z, r, z): H_{\om, \ep} (p_r, p_z, r, z) = h \right\}
\end{equation*}
a three-dimensional manifold. This configuration exhibits the following fundamental properties 

\begin{prop}\label{prop;diffeomorphic to S3}
For any $\ep\ge 0$ small enough, there is a continuous family of subsets $\Gamma_{\ep} \subset \rr \setminus \{0\} \times \rr$ with 
$$ \Gamma_0 = \{ (\om,h) \in \rr \setminus \{0\} \times \rr:-1<2h\om^2<0 \},$$
such that, when $(\om,h)\in\Gamma_\ep$, the following results hold. 
	\begin{enumerate}
		\item[(a)] $\mathcal{M}(h,\om,\ep)$ forms a compact contact manifold diffeomorphic to $\mathbb{S}^{3}$ endowed with contact form $\lambda$.
		\item[(b)] There exists $R$ close to $\om^2/2$ satisfying $\min\Pi_{r}\mathcal{H}(h,\om,\ep)>R$, where $\Pi_{r}$ denotes the projection to the $r$-axis and
		\begin{equation*}
			\mathcal{H}(h, \om, \ep) = \Pi_{r, z} \mathcal{M}(h, \om, \ep) = \left\{ (r, z): (p_r, p_z, r, z) \in \mathcal{M}(h, \om, \ep)\right\}.
		\end{equation*}		
	\end{enumerate}
\end{prop}

\begin{remark}
	Here $\mathcal{H}(h, \om, \ep)$ represents Hill's region. In most of the previous studies of the perturbed Kepler problem, the collision singularity is usually regularized following Levi-Civita \cite{LC20}, Kustaanheimo-Stiefel \cite{KS65} or Moser \cite{Moser70}. As a result, it is unclear whether the obtained solutions contain collisions or not, for example in \cite{BOZ19} and \cite{Zhao23}. Meanwhile, in our case, the central body in the ellipsoid problem is not a point mass. The avoidance of collisions is crucial. Property (b) guarantees orbital avoidance of the central body for sufficiently large angular momentum $\om$.
\end{remark}

This proposition enables the application of symplectic and contact geometric methods to analyze mechanical systems. Recent developments in this direction include notable results for the restricted three-body problem \cite{AFvKP12,AFFHvK12,MvK22} and the isosceles three-body problem \cite{Hu2023ASD}. 

The foundational concept of global surfaces of section originated with Poincaré's work \cite{Pc12} on the circular planar restricted three-body problem. His fundamental theorem on area-preserving twist maps of the annulus, later proved by Birkhoff \cite{Bk13}, became known as the Poincaré-Birkhoff Theorem. Franks \cite{Franks92} later generalized this result, showing that the existence of one periodic point implies infinitely many periodic points on the annulus. This framework extends to contact manifolds $(\mathcal{M},\lambda)$ diffeomorphic to $\mathbb{S}^{3}$: If the Reeb flow admits a disk-like global surface of section, the system must contain either two or infinitely many periodic orbits. Recent breakthrough work by Cristofaro-Gardiner, Hryniewicz, Hutchings, and Liu \cite{CHHL23} revealed an algebraic relationship between the contact volume and the action and Seifert rotation number of Reeb orbits (defined in Sections \ref{sec:KP} and \ref{sec:KPSP}) for contact manifolds with exactly two Reeb orbits. Crucially, violation of this relationship guarantees the existence of infinitely many periodic orbits. In our setting, the reflection symmetry condition in \eqref{eq;symm-cond} ensures a special periodic solution $\xi_{\ep}(t)$ entirely contained in the $(p_r, r)$-plane that bounds a global surface of section (see Section \ref{sec:KPSP}). Through the calculations of this orbit's action and Seifert rotation number combined with the algebraic relationship, we obtain:
\begin{thm}\label{thm;Exist infinity many 1}
For sufficiently small $\ep$ and parameters $(\om,h) \in \Gamma_\ep$, the reduced Hamiltonian system \eqref{eq:spacial kepler pertur Ham} admits infinitely many periodic orbits when one of the two conditions in Theorem \ref{thm:Exist infinity many} is satisfied. 
\end{thm}

While the above result proves the existence of infinitely many periodic orbits, it requires some additional conditions. Meanwhile, we can show the existence of at least one special periodic orbit as long as the perturbation is small enough. 
To explain this, we define two types of periodic solutions as follows.  
\begin{defi} A periodic solution $\xi(t)=(p_{r}(t),p_{z}(t),r(t),z(t))$ of the reduced Hamiltonian system 
\begin{enumerate}
\item[(a)]  is \textbf{a brake orbit} if there exist two distinct time moments $t_1 \ne t_2$, such that the momentum variables all vanish, i.e., $p_r(t_i)=p_z(t_i)=0$, for $i =1,2$;
\item[(b)] is \textbf{a $z$-symmetric orbit} if $p_r(t_0)=z(t_0)=0$ for some $t_{0}\in\mathbb{R}$ and
\begin{equation}
	\xi(t_{0}+t)=\left({\begin{array}{cc}
			A & 0\\
			0 & -A\\
	\end{array}}\right)\xi(t_{0}-t),\quad \forall t\in\mathbb{R}, \text{ where } A=\left({\begin{array}{cc}
		-1 & 0\\
		0 & 1\\
\end{array}}\right).
\end{equation} 
\end{enumerate}
\end{defi} 
\begin{remark}
A brake orbit must be a periodic solution, and when it is also $z$-symmetric, we say it is a $z$-symmetric brake orbit.
\end{remark}
 
\begin{thm}\label{thm:z-sym exist}
	For sufficiently small $\ep$ and parameters $(\om,h) \in \Gamma_\ep$, there exists a unique $z$-symmetric brake orbit $\xi_{bz,\ep}(t)$ forming a Hopf link with $\xi_{\ep}(t)$ in $\mathcal{M}(h,\om,\ep)$.
\end{thm}
 

Our results extend beyond the ellipsoid problem to a special $(1+n)$-body configuration called \textbf{the $n$-pyramidal problem}. After fixing the center of mass at the origin, there is an invariant subsystem, where one of the masses $m_0$ always moves on the $z$-axis and the remaining $n$ equal masses $m$ always form a regular $n$-gon with the plane containing them always perpendicular to the $z$-axis. Then the movement of one of the masses forming the $n$-gon can represent the movement of the whole system. Let $\ep = m/m_0$ denote the mass ratio. Then for energy $h<0$ and $z$-axis angular momentum $\om\neq0$, such that the corresponding energy surface $\M(h,\om,\ep)$ is compact, we have the following theorem.

\begin{thm}\label{thm:The isosceles three body problem}
	The $n$-pyramidal problem admits infinitely many relative periodic solutions within $\M(h,\om,\ep)$, when the mass ratio $\ep\in(0,\ep_{0}(n))$ with sufficiently small $\ep_{0}(n)$ and $2\leq n\leq472$.
\end{thm}
\begin{remark}
	When $n=2$, the $n$-pyramidal problem is the well-known isosceles three body problem, which has been studied by many authors. Moeckel \cite{Mk84} characterized collision asymptotics in this system. Using Maslov-type index (also called Conley-Zehnder index) and contact geometry, the first and third authors and their collaborators \cite{Hu2023ASD} established a criterion for the existence of infinitely many relative periodic orbits. Grotta-Ragazzo, Liu and Salom{\~a}o \cite{GLSarxiv} obtained infinitely many relative periodic orbits when the eccentricity of Euler orbit is small enough using Birkhoff normal form. Our theorem solves the problem for the eccentricity of the Euler orbit $e\in(0,1)$ but the mass ratio $\ep$ is small enough.
\end{remark}
\begin{remark}
	For technical reasons, the above theorem does not apply when $n > 472$. However, we believe the same result should still hold.
\end{remark}

The paper is organized as follows. In Section \ref{sec:KP}, we introduce the reduced Kepler problem  and obtain some results which are needed in the perturbed cases. In Section \ref{sec:KPSP}, the perturbed system is studied and proofs of Theorem \ref{thm;Exist infinity many 1} and Theorem \ref{thm:z-sym exist} are given. In Section \ref{sec:ellipsoid}, we introduce the ellipsoid problem and prove Theorem \ref{thm:axial symmetry ellipsoid potential}. In Section \ref{sec:pyramidal}, we introduce the $n$-pyramidal problem and prove Theorem \ref{thm:The isosceles three body problem}.

\section{The Kepler Problem}\label{sec:KP}

In this section, we obtain some results of the Kepler problem that will be needed in the perturbed cases. Recall that in Cartesian coordinate $q=(x,y,z)\in\mathbb{R}^{3}$ with the corresponding momentum variable $p=(\dot{x},\dot{y},\dot{z})\in\mathbb{R}^{3}$, the Hamiltonian of the Kepler problem can be written as		    
\begin{equation}
	H(p,q)=\frac{|p|^2}{2}-\frac{1}{|q|}.
	\label{eq:spacial kepler Ham1}
\end{equation}
Meanwhile, in cylindrical coordinates $(r, \tht, z)$ with $(p_r,p_{\tht},p_z)=(\dot{r},r^2\dot{\tht},\dot{z})$, the above Hamiltonian can be written correspondingly as
\begin{equation*}
	H(p_r,p_{\tht},p_z,r,\tht,z)=\frac{1}{2}(p_{r}^{2}+r^{-2}p_{\tht}^2+p_{z}^{2})-\frac{1}{\sqrt{r^{2}+z^{2}}}.
\end{equation*} 

Notice that $p_{\tht} = r^2 \dot{\tht}$ is the angular momentum with respect to the $z$-axis, which is a first integral, as the system is symmetric with respect to the $z$-axis. By fixing the angular momentum to be $\om$, we get the following reduced Hamiltonian
\begin{equation}\label{eq:spacial kepler Ham2}
	H_{\om}(p_r,p_z,r,z)=\frac{1}{2}(p_{r}^{2}+p_{z}^{2})+\frac{\om^2}{2r^2}-\frac{1}{\sqrt{r^{2}+z^{2}}},
\end{equation}
which is a system with two degrees of freedom. If $h$ is the energy constant, we get the three-dimensional energy surface 
$$ \M(h, \om) = \{ (p_r, p_z, r, z): H_{\om} (p_r, p_z, r, z)= h \}.$$ 
Let $\Pi_{r,z}$ denote the projection to the $(r, z)$-plane, we have Hill's region as follows:
$$ \mathcal{H}(h, \om) = \Pi_{r, z} \M(h, \om)= \{ (r, z): (p_r, p_z, r, z) \in \M(h, \om)\}. $$
\begin{prop}\label{prop;diffeomorphic to S3 1}
	For the reduced Hamiltonian system \eqref{eq:spacial kepler Ham2},  
	\begin{enumerate}
		\item[(a)] if $2h\om^2<-1$, then $\mathcal{M}(h,\om)=\varnothing$; 
		\item[(b)] if $2h\om^2=-1$, then $\mathcal{M}(h,\om)=\{(0,0,\om^2,0)\}$;
		\item[(c)] if $-1<2h\om^2<0$, then $\mathcal{M}(h,\om)$ is a compact manifold diffeomorphic to $\mathbb{S}^{3}$ with
		$$ \min\{ r: (p_r, p_z, r,z ) \in \mathcal{M}(h, \om) \}=\frac{\om^2}{1+\sqrt{1+2h\om^2}}>\frac{\om^2}{2}; $$
		\item[(d)] if $2h\om^2>0$, then $\mathcal{M}(h,\om)$ is unbounded.
	\end{enumerate}
\end{prop}
\begin{proof}
	By a direct computation, 
	\begin{equation*}
		\nabla H_{\om}=\left(p_{r}, p_{z},-\frac{\om^{2}}{r^{3}}+\frac{r}{(r^{2}+z^{2})^{3\slash2}},\frac{z}{(r^{2}+z^{2})^{3\slash2}}\right).
	\end{equation*}
	Then $(0,0,\om^2,0)$ is the unique critical point of $H_{\om}$.
	
	Meanwhile the Hessian matrix at $(0,0,\om^2,0)$ is
	\begin{equation*}
		D^{2}H_{\om}(0,0,\om^2,0)=\left({\begin{array}{cccc}
				1 & 0 & 0 & 0\\
				0 & 1 & 0 & 0\\
				0 & 0 & \frac{1}{\om^6} & 0\\
				0 & 0 & 0 & \frac{1}{\om^6}
		\end{array}}\right),
	\end{equation*}
	which is positive definite. As a result, $ -\frac{1}{2\om^2} =H_{\om}(0,0,\om^2,0)$ is the minimum of $H_{\om}$. Then properties (a) and (b) follow from this immediately.  

	Notice that $H_{\om}$ is increasing with respect to $z^2$ and
	$$ \lim_{z \to \infty}H_\om(p_r,p_z,r,z)=\frac{1}{2}(p_{r}^{2}+p_{z}^{2})+\frac{\om^2}{2r^2}.$$
Therefore $\mathcal{M}(h,\om)$ is bounded in $z$-direction if and only if
$$-\frac{1}{2\om^2}< h< \inf  \left\{ \frac{1}{2}(p_{r}^{2}+p_{z}^{2})+\frac{\om^2}{2r^2} \right\}=0.$$
Then $\mathcal{M}(h,\om)$ is unbounded when $2h\om^2>0$. This proves property (d). 

For property (c), we only need to show $\mathcal{M}(h,\om)$ is bounded in $r$-direction, when $-1<2h\om^2<0$. Consider the boundary of $\mathcal{H}(h,\om)$: 
\begin{equation}
	V_{\om}(r,z)=\frac{\om^{2}}{2r^{2}}-\frac{1}{\sqrt{r^{2}+z^{2}}}=h.
\end{equation}
Because $V_{\om}(r,z)$ is increasing with respect to $z^2$, so we only need to consider $z=0$. Solving the equation
\begin{equation}
	\frac{\om^{2}}{2r^{2}}-\frac{1}{r}=h,
\end{equation}
we have $r^{\pm}=\frac{\om^2}{1\pm\sqrt{1+2h\om^2}}>\frac{\om^2}{2}$. So we have $\Pi_{r}\mathcal{H}(h,\om)=[r^+,r^-]$, which is bounded. Therefore $\mathcal{M}(h,\om)$ is compact. Because the unique critical point $(0,0,\om^2,0)$ of $H_\om$ is nondegenerate, when $h$ is slightly greater than $-\frac{1}{2\om^2}$, the energy surface $\mathcal{M}(h,\om)$ is diffeomorphic to $\mathbb{S}^3$. When $-1<2h\om^2<0$, the regular energy surfaces $\mathcal{M}(h,\om)$ are diffeomorphic to one another. We can also use the Reeb sphere theorem to prove $\mathcal{M}(h,\om)$ is diffeomorphic to $\mathbb{S}^{3}$. The Morse function $f(p_r,p_z,r,z)=r$ has exactly two critical points $(0,0,r^+,0)$ and $(0,0,r^-,0)$ on $\mathcal{M}(h,\om)$. Combining together with the uniqueness of the smooth structure on $\mathcal{M}(h,\om)$, $\mathcal{M}(h,\om)$ is diffeomorphic to $\mathbb{S}^{3}$.
	
\end{proof}

For simplicity, we denote $\mathcal{M}(h,\om)$ and $\mathcal{H}(h,\om)$ by $\mathcal{M}$ and $\mathcal{H}$ respectively below, when there is no confusion. Without loss of generality, let's assume $\om>0$ in the rest of this section. 

The Hamiltonian vector field $X_{H_\om}$ is determined by $\omega(\cdot,X_{H_\om})=\mathrm{d}H_\om$, where $\omega=\mathrm{d}p_{r}\wedge\mathrm{d}r+\mathrm{d}p_{z}\wedge\mathrm{d}z$ is the symplectic form. Since $H_\om$ is a mechanical Hamiltonian, the hypersurface $\M$ is of contact type. That is, there exists a 1-form $\lambda$ on $\M$ satisfying $\lambda\wedge\mathrm{d}\lambda\neq0$, so that its Reeb vector field $R$, determined by $\mathrm{d}\lambda(R,\cdot)\equiv 0$ and $\lambda(R)=1$, is parallel to $X_{H_\om}$ and $\mathrm{d}\lambda=\omega$. 

From now on we will focus on the cases with $2h\om^2 \in (-1, 0)$. By Proposition \ref{prop;diffeomorphic to S3 1}, $\mathcal{M}$ is a compact manifold diffeomorphic to $\mathbb{S}^{3}$. Following the definitions in \cite{CHHL23}, we define its contact volume as 
\begin{equation}
\text{Vol}(\mathcal{M})=\int_{\mathcal{M}}\lambda\wedge\mathrm{d}\lambda,
\end{equation}
and using Stokes' formula 
\begin{equation} \label{eq:the expression of contact volume}
	\text{Vol}(\mathcal{M}) =\int_{S}\mathrm{d}\lambda\wedge\mathrm{d}\lambda=2\int_{S}\mathrm{d}p_{r}\wedge\mathrm{d}r\wedge\mathrm{d}p_{z}\wedge\mathrm{d}z,
\end{equation}
where $S$ is the region bounded by $\M$, i.e.,  $\partial S=\mathcal{M}$. 

For any periodic solution $\xi(t)$ of the reduced Kepler problem \eqref{eq:spacial kepler Ham2}, define its action value as 
\begin{equation} 
	A(\xi(t))=\int_{\xi(t)}\lambda.
\end{equation}
For the given energy $h$ and angular momentum $\om$, we can find a unique solution $\xi_0(t)$ of the reduced Hamiltonian system \eqref{eq:spacial kepler Ham2} satisfying 
\begin{equation} \label{eq;r-tht}
	r(t)=\frac{\om^{2}}{1+e\cos\theta(t)}, \text{ where } e=\sqrt{1+2h\om^{2}}.
\end{equation}
It corresponds to an elliptic orbit of the Kepler problem with eccentricity $e$ and entirely contained in the $(x,y)$-plane. 
As a result, its action value is
\begin{equation} \label{eq:the expression of action}
	A(\xi_0(t))=\int_{\xi_0(t)}\lambda =\int_{\Upsilon}\mathrm{d}\lambda=\int_{\Upsilon}\mathrm{d}p_{r}\wedge\mathrm{d}r,
\end{equation}
where $\Upsilon$ is the region contained in the $(p_r, r)$-plane and bounded by $\xi_0(t)$, i.e., $\partial \Upsilon=\xi_0(t)$.

\begin{lemma}\label{lemma:kepler equal}
	 When $-1<2h\om^2<0$, $\text{Vol}(\mathcal{M})=A^{2}(\xi_0(t))$.
\end{lemma}
\begin{proof}
	Notice that $\Upsilon$, $S$ and the Hill's region $\mathcal{H}$ can be given as
	\begin{equation}
		\begin{aligned}
			&\Upsilon=\left\{(p_{r},r)|\frac{1}{2}p_{r}^{2}+\frac{\om^{2}}{2r^{2}}-\frac{1}{r}\leq h \right\},\\
			&S=\left\{(p_{r},p_{z},r ,z)|\frac{1}{2}(p_{r}^{2}+p_{z}^{2})+\frac{\om^{2}}{2r^{2}}-\frac{1}{\sqrt{r^{2}+z^{2}}}\leq h \right\},\\
			&\mathcal{H}=\left\{(r,z)|\frac{\om^{2}}{2r^{2}}-\frac{1}{\sqrt{r^{2}+z^{2}}}\leq h \right\}.
		\end{aligned}
	\end{equation}
	Recall that $r^{\pm}$ are the roots of $\frac{\om^{2}}{2r^{2}}-\frac{1}{r}=h$. Using \eqref{eq;r-tht}, we can get
	\begin{equation}\label{eq:action of Kepler}
		\begin{aligned}
			A(\xi_{0}(t))&=\int_{\Upsilon}\mathrm{d}p_{r}\wedge\mathrm{d}r=2\int_{r^{+}}^{r^{-}}\sqrt{2h-\frac{\om^{2}}{r^{2}}+\frac{2}{r}}\mathrm{d}r\\
			&=2\int_{0}^{\pi}\frac{e\sin\theta}{\om}\cdot\frac{\om^{2}e\sin\theta}{(1+e\cos\theta)^{2}}\mathrm{d}\theta=2e^{2}\om\int_{0}^{\pi}\frac{\sin^{2}\theta}{(1+e\cos\theta)^{2}}\mathrm{d}\theta\\
			&=2\om\left(\int_{0}^{\pi}\frac{1}{1+e\cos\theta}\mathrm{d}\theta-\pi\right)=2\pi\om\left(\frac{1}{\sqrt{1-e^{2}}}-1\right).
		\end{aligned}
	\end{equation}

	Meanwhile under polar coordinates $(r,z)=(\rho\cos\varphi,\rho\sin\varphi)$, Hill's region $\mathcal{H}$ can be given as
    $$\mathcal{H}=\{(\rho,\varphi)|\rho^-(\varphi)\leq\rho\leq\rho^+(\varphi), \varphi^-\leq\varphi\leq\varphi^+\}=\{(\rho,\varphi)|\varphi^-(\rho)\leq\varphi\leq\varphi^+(\rho), \rho^-\leq\rho\leq\rho^+\},$$
    where 
	$$\begin{aligned}
		&\rho^{\pm}(\varphi)=\frac{1\pm\sqrt{1+2h\om^2/\cos^2\varphi}}{-2h},\qquad \quad \varphi^{\pm}=\pm\arccos\sqrt{-2h\om^2},\\
		&\varphi^{\pm}(\rho)=\pm\arctan\sqrt{\frac{2h\rho^2+2\rho-\om^{2}}{\om^2}},\quad \rho^{\pm}=\frac{1\pm\sqrt{1+2h\om^2}}{-2h}.
	\end{aligned}$$
    Using this we can compute the contact volume of $\mathcal{M}$ as below. 
	\begin{equation}
		\begin{aligned}
			\text{Vol}(\mathcal{M})&=2\int_{S}\mathrm{d}p_{r}\wedge\mathrm{d}r\wedge\mathrm{d}p_{z}\wedge\mathrm{d}z\\
			&=2\pi\iint_{\mathcal{H}}2h-\frac{\om^{2}}{r^{2}}+\frac{2}{\sqrt{r^{2}+z^{2}}}\ \mathrm{d}r\mathrm{d}z\\
			&=2\pi\iint_{\mathcal{H}}2h\rho-\frac{\om^{2}}{\rho\cos^{2}\varphi}+2\ \mathrm{d}\rho\mathrm{d}\varphi\\
			&=2\pi\left(\int_{\varphi^-}^{\varphi^+}\int_{\rho^-(\varphi)}^{\rho^+(\varphi)}2h\rho+2\ \mathrm{d}\rho \mathrm{d}\varphi-\int_{\rho^-}^{\rho^+}\int_{\varphi^-(\rho)}^{\varphi^+(\rho)}\frac{\om^2}{\rho\cos^2\varphi} \mathrm{d}\varphi \mathrm{d}\rho\right)\\
			&=2\pi\left(I_{1}-I_{2}\right)
		\end{aligned}
	\end{equation}
	To compute $I_1$ and $I_2$, we need the following formulas, when $a\in(0,1)$,
	\begin{equation}\label{formula 1}
		\begin{aligned}
			\int_{-\pi\slash2}^{\pi\slash2}\frac{\cos^2\theta}{\cos^2\theta+a\sin^{2}\theta}\mathrm{d} \theta&=\int_{-\pi\slash2}^{\pi\slash2}\frac{1}{1+a\tan^{2}\theta}\mathrm{d} \theta=\int_{-\infty}^{+\infty}\frac{1}{1+at^2}\cdot\frac{1}{1+t^2}\mathrm{d}t\\
			&=\int_{-\infty}^{+\infty}-\frac{a}{1-a}\cdot\frac{1}{1+at^2}+\frac{1}{1-a}\cdot\frac{1}{1+t^2}\mathrm{d}t\\
			&=\left(-\frac{\sqrt{a}}{1-a}\arctan\sqrt{a}t+\frac{1}{1-a}\arctan t \right)\bigg|^{+\infty}_{-\infty}=\frac{\pi}{1+\sqrt{a}};
		\end{aligned}
	\end{equation}
	\begin{equation} \label{formula 2} \begin{aligned}
	\int_{-\pi\slash2}^{\pi\slash2}\frac{\cos^2\theta}{1+a\sin\theta}\mathrm{d} \theta&=\int_{-\pi\slash2}^{\pi\slash2}\frac{\cos^2\theta}{1-a^2\sin^{2}\theta}\mathrm{d} \theta-\int_{-\pi\slash2}^{\pi\slash2}\frac{a\cos^2\theta\sin\theta}{1-a^2\sin^{2}\theta}\mathrm{d} \theta\\
			&=\int_{-\pi\slash2}^{\pi\slash2}\frac{\cos^2\theta}{\cos^{2}\theta+(1-a^2)\sin^{2}\theta}\mathrm{d} \theta=\frac{\pi}{1+\sqrt{1-a^2}}.
	\end{aligned}		
	\end{equation}

	Since $\cos^2\varphi^+=-2h\om^2=1-e^2$ and $\sin^2\varphi^+=1+2h\om^2=e^2$, we choose $\sin\varphi=e\sin\theta$, and we conclude that $\cos\varphi\mathrm{d}\varphi=e\cos\theta\mathrm{d}\theta$. We compute that
	$$
	\begin{aligned}
		I_1&=\int_{\varphi^-}^{\varphi^+}\int_{\rho^-(\varphi)}^{\rho^+(\varphi)}2h\rho+2 \mathrm{d}\rho \mathrm{d}\varphi=\frac{1}{-h}\int_{\varphi^-}^{\varphi^+}\left(1+\frac{2h\om^{2}}{\cos^2\varphi}\right)^{1/2}\mathrm{d} \varphi\\
		&=\frac{1}{-h}\int_{\varphi^-}^{\varphi^+}\frac{(e^2-\sin^{2}\varphi)^{1/2}}{\cos^2\varphi}\mathrm{d}\sin \varphi=\frac{e^2}{-h}\int_{-\pi/2}^{\pi/2}\frac{\cos^{2}\theta}{1-e^2\sin^2\theta}\mathrm{d} \theta\\
		&=\frac{e^2}{-h}\int_{-\pi\slash2}^{\pi\slash2}\frac{\cos^2\theta}{\cos^2\theta+(1-e^2)\sin^{2}\theta}\mathrm{d} \theta=\frac{e^2}{-h}\cdot\frac{\pi}{1+\sqrt{1-e^2}}\\
		&=\frac{2\pi\om^2}{\sqrt{1-e^2}}\cdot\left(\frac{1}{\sqrt{1-e^{2}}}-1\right). 
	\end{aligned}$$
	For $I_2$, we choose $u=\rho+1/(2h)=u^+\sin\theta$, we have $u^\pm=\rho^\pm+1/(2h)=\pm e/(-2h)$ and we compute that 
	$$
	\begin{aligned}
		I_2&=\int_{\rho^-}^{\rho^+}\int_{\varphi_-(\rho)}^{\varphi^+(\rho)}\frac{\om^2}{\rho\cos^2\varphi} \mathrm{d}\varphi \mathrm{d}\rho=2\om\int_{\rho^-}^{\rho^+}\frac{\sqrt{2h\rho^2+2\rho-\om^{2}}}{\rho}\mathrm{d} \rho\\    &=2\sqrt{1-e^2}\int_{u^-}^{u^+}\frac{\sqrt{e^2/(4h^2)-u^{2}}}{u+1/(-2h)}\mathrm{d} u=\sqrt{1-e^2}\cdot\frac{e^2}{-h}\int_{-\pi\slash2}^{\pi\slash2}\frac{\cos^2\theta}{1+e\sin\theta}\mathrm{d} \theta\\
		&=2\pi\om^2\left(\frac{1}{\sqrt{1-e^{2}}}-1\right).
	\end{aligned}$$
	The above computation implies $\text{Vol}(\mathcal{M}) =A^2(\xi_0(t))$. 
\end{proof}
We now study some special orbits, including brake orbits and $z$-symmetric orbits. A definition of a global surface of section is required.
\begin{defi}
	Let $X$ be a non-vanishing vector field on $\mathbb{S}^{3}$ with flow $\varphi_{t}^{X}$. A disk-like global surface of section is an embedded closed 2-dimensional disk $D\subset\mathbb{S}^{3}$ satisfying
	\begin{enumerate}
		\item[(1)] $X$ is tangent to $\partial D$, the boundary of $D$;
		\item[(2)] $X$ is transverse to $D^{\circ}$, the interior of $D$;
		\item[(3)] for any $x\in\mathbb{S}^{3}\setminus\partial D$, there exist $t^{+}>0$ and $t^{-}<0$ such that both $\varphi_{t^{+}}^{X}(x)$ and $\varphi_{t^{-}}^{X}(x)$ are contained in the interior of $D$.
	\end{enumerate}
\end{defi} 
\begin{lemma}\label{lemma:global serface of section 1}
	When $2h \om^2 \in (-1, 0)$, $\Sigma_{\pm}$ given below are two disk-like global surfaces of section of the Hamiltonian flow $\varphi_{t}$ corresponding to the Hamiltonian vector field $X_{H_{\om}}$,
	\begin{equation*}
	\begin{aligned}
		&\Sigma_{+}=\{(p_{r},p_{z},r,z)\in\mathcal{M}:z=0,p_{z}\geq0\},\\
		&\Sigma_{-}=\{(p_{r},p_{z},r,z)\in\mathcal{M}:z=0,p_{z}\leq0\}.
	\end{aligned}
\end{equation*}
\end{lemma}
\begin{proof}
	The boundary $\partial \Sigma_{\pm }=\{z=0,p_{z}=0\}\cap \mathcal{M}(h,\om)$ coincide with the orbit $ \xi_0(t) $, which is tangent to $ X_{H_\om} $.
	
	For $\xi=(p_{r},p_{z},r,z)\in\Sigma_{+}^{\circ}$, the last component of $X_{H_{\om}}(\xi)=J\nabla H_{\om}(\xi)$ is $ p_z > 0 $, ensuring transversality to $ \Sigma_+^\circ $. 
	
	Consider a trajectory $\xi(t)=(p_{r}(t),p_{z}(t),r(t),z(t))$ starting from $\Sigma_{+}^{\circ}$. Suppose no $ t_1 > 0 $ exists with $ p_z(t_1) = 0 $ and $ z(t_1) > 0 $. Then $ z(t) $ monotonically increases to some $ z_0 > 0 $, and the $ \omega $-limit set of $ \xi(t) $ lies in $\{z=z_{0}\}\cap \mathcal{M}$. By the invariance of the $ \omega $-limit set under $\varphi_{t}$, there exists a solution $\hat{\xi}(t)=(\hat{p}_{r}(t),\hat{p}_{z}(t),\hat{r}(t),\hat{z}(t))$ with $\hat{z}(t)=z_{0}>0$ and $\hat{p}_{z}(t)=\dot{\hat{z}}(t)=0$. However,  
	\begin{equation*}
		\dot{\hat{p}}_{z}(t)=\ddot{\hat{z}}(t)=-\frac{\hat{z}(t)}{(\hat{r}^{2}(t)+\hat{z}^{2}(t))^{3\slash 2}}<0,
	\end{equation*}
	yielding a contradiction. Thus, $ t_1 $ exists. 
	Since $\ddot{z}(t)<-cz(t)$ for some $c>0$, by the Sturm-Picone Comparison Theorem, there exists $t_{+}>t_{1}$ with $ p_z(t_+) > 0 $ and $ z(t_+) = 0 $, proving $ \Sigma_+ $ is a global surface of section. The argument for $ \Sigma_- $ is analogous.
\end{proof}

Let  $\mathcal{B}_{+}=\{(r,z)\in\partial\mathcal{H}:z>0\}$ and $\mathcal{B}_{-}=\{(r,z)\in\partial\mathcal{H}:z<0\}$ denote the upper and lower Hill’s boundaries. For a solution $\xi(t)=(p_{r}(t),p_{z}(t),r(t),z(t))$ with $\Pi_{r,z}\xi(0)\in\mathcal{B}_{+}$, Lemma \ref{lemma:global serface of section 1} guarantees a minimal $t_{1}>0$ such that $\xi(t_{1})\in\Sigma_-$.
 
Expressing $\xi(t)$ in the original Hamiltonian system \eqref{eq:spacial kepler Ham1} via spherical coordinates $ (\rho, \theta, \varphi) $, we have  
\begin{equation}\label{eq:Kepler explicitly}
	\rho(t)=\frac{|\Omega|^{2}}{1+\hat{e}\cos{(\hat{\theta}(t)+g)}},
\end{equation}   
where $\Omega=q\times\dot{q}=(\om_{1},\om_{2},\om_{3})\in\mathbb{R}^{3}$ is the angular momentum vector, $\hat{e}=\sqrt{1+2h|\Omega|^{2}}$ is the eccentricity, $\hat{\theta}(t)$ which satisfies $\rho^2(t)\mathrm{d}\hat{\theta}/\mathrm{d}t=|\Omega|$ and $\hat{\theta}(0)=0$ is the true anomaly and $g$ is the argument of perigee. The third component $ \om = \om_3 $ corresponds to the angular momentum about the $ z $-axis.

For $\Pi_{r,z}\xi(0)\in\mathcal{B}_{+}$, the condition $\dot{\rho}(0)=0$ implies $ g \in \{0, \pi\} $. \\ 
(a) If $g=0$ and $\hat{\theta}(t_2)=\pi$ for some $t_2>t_1$, then
$\Pi_{r,z}\xi(t_2)\in\mathcal{B}_{-}$. Here, $ \xi(t) $ moves from pericenter to apocenter.\\  
(b) If $g=\pi$ and $\hat{\theta}(t_2)=\pi$ for some $t_2>t_1$, then $\Pi_{r,z}\xi(t_2)\in\mathcal{B}_{-}$. Here, $ \xi(t) $ moves from apocenter to pericenter. 
	
In both cases, $p_z(t)<0$ for $t\in(0,t_2)$, and the plane formed by the major axis direction of the elliptic orbit and angular momentum direction contains $z$-axis. When $\xi(t)$ intersects the $r$-axis (i.e., at $t=t_1$), $ \hat{\theta}(t_1) = \pi/2 $. Thus, every trajectory starting from Hill’s boundaries is a brake orbit.
	
For a solution $\bar{\xi}(t)=(p_{r}(t),p_{z}(t),r(t),z(t))$ with $\bar{\xi}(0)\in\Sigma_+\cap\{p_r=0\}$, Lemma \ref{lemma:global serface of section 1} ensures a minimal $t_{1}>0$ such that $\bar{\xi}(t_{1})\in\Sigma_-$. Since $\bar{\xi}(t)$ is an elliptic orbit, symmetry implies that its major axis lies on the $(x,y)$-plane and $\bar{\xi}(t_1)\in\Sigma_-\cap\{p_r=0\}$. Consequently, $\bar{\xi}(t)$ is a $z$-symmetric periodic orbit.
	
\begin{lemma}\label{lemma:relationship pr r}
	For a brake orbit $\xi(t) = (p_{r}(t), p_{z}(t), r(t), z(t))$ with $\Pi_{r,z}\xi(0)\in\mathcal{B}_+$, we have
	$$
	p_{r}(t_1) = \frac{\om^2 + 2hr^2(0)}{2\om r(0)}.
	$$
\end{lemma}

\begin{proof}
	Interpreting $\xi(t)$ as a solution of the planar Kepler problem,  
	\begin{equation}
		H(p_{\rho}, p_{\hat{\theta}}, \rho, \hat{\theta}) = \frac{1}{2} \left( p_{\rho}^2 + \frac{p_{\hat{\theta}}^2}{\rho^2} \right) - \frac{1}{\rho}, \label{eq:energy 1}
	\end{equation}
	and in spherical coordinates $(\rho, \theta, \varphi)$,  
	\begin{equation}
		H(p_{\rho}, p_{\theta}, p_{\varphi}, \rho, \theta, \varphi) = \frac{1}{2} \left( p_{\rho}^2 + \frac{p_{\theta}^2}{\rho^2 \cos^2 \varphi} + \frac{p_{\varphi}^2}{\rho^2} \right) - \frac{1}{\rho}, \label{eq:energy 2}
	\end{equation}
	angular momentum conservation yields  
	\begin{equation}
		|\Omega| = |p_{\hat{\theta}}|, \quad \om = p_{\theta}, \label{eq:angular momentum conservation}
	\end{equation}
	where $\Omega$ is the total angular momentum vector. Energy conservation implies  
	\begin{equation}
		\om^2 \cos^{-2}\varphi + p_{\varphi}^2 = |\Omega|^2. \label{eq:relationship of omega}
	\end{equation}
	Let $\xi(0) = (p_r, p_z, r, z)$ and $\xi(t_1) = (\tilde{p}_r, \tilde{p}_z, \tilde{r}, \tilde{z})$. Obviously, $p_r=p_z=\tilde{z}=0$. At Hill’s boundary, the condition  
	$$
	\frac{\om^2}{2r^2} - \frac{1}{\sqrt{r^2 + z^2}} = h,  
	$$  
	yields  
	$$
	z^2 = r^2 \left( \left( \frac{2r}{\om^2 - 2hr^2} \right)^2 - 1 \right).  
	$$  
	For $(r, z) \in \mathcal{B}_{+}$, express $r = \rho \cos\varphi$ and $z = \rho \sin\varphi$. Then $\dot{\varphi} = 0$, and  
	$$
	\tan^2\varphi = \frac{z^2}{r^2} = \left( \frac{2r}{\om^2 - 2hr^2} \right)^2 - 1.  
	$$  
	From \eqref{eq:relationship of omega} and $p_{\varphi} = \rho^2 \dot{\varphi} = 0$, we obtain  
	\begin{equation}
		\cos^2\varphi = \frac{\om^2}{|\Omega|^2}, \quad \frac{\om}{|\Omega|} = \frac{\om^2 - 2hr^2}{2r}. \label{eq:omega ratio}
	\end{equation}
	For $\xi(t_1)$, also express $\tilde{r} = \tilde{\rho} \cos\tilde{\varphi}$ and $\tilde{z} = \tilde{\rho} \sin\tilde{\varphi}$, $\tilde{z} = 0$ implies $\tilde{\varphi} = 0$. By \eqref{eq:Kepler explicitly} and $\cos\hat{\theta}(t_1) = 0$,  
	\begin{equation}
			\tilde{r} = |\Omega|^2, \quad \tilde{p}_r = \dot{\tilde{\rho}} = \tilde{p}_\rho. \label{eq:tilde relations}
	\end{equation} 
	Energy conservation \eqref{eq:energy 1} and angular momentum conservation \eqref{eq:angular momentum conservation} yield  
	\begin{equation}
		\frac{1}{2} \left( \tilde{p}_r^2 + \frac{|\Omega|^2}{\tilde{\rho}^2} \right) - \frac{1}{\tilde{\rho}} = h. \label{eq:energy conservation}
	\end{equation}  
	Combining \eqref{eq:omega ratio}, \eqref{eq:tilde relations}, and \eqref{eq:energy conservation}, we have  
	$$
	\tilde{p}_r^2 = \left( \frac{\om^2 + 2hr^2}{2\om r} \right)^2.  
	$$  
	To determine the sign of $\tilde{p}_r$, observe that near the pericenter, $\tilde{\rho}(t)$ increases. Thus, $\tilde{p}_r = \tilde{p}_\rho > 0$, giving 
	$$
	\tilde{p}_r = \frac{\om^2 + 2hr^2}{2\om r}.  
	$$  
\end{proof}

\begin{lemma}\label{lemma;the existence of bz for Kepler}
	For $-1<2h\om^2<0$, there exists one and only one $z$-symmetric brake orbit $\xi_{bz}(t)$ besides $\xi_0(t)$ inside $\mathcal{M}$, moreover, $\xi_{bz}(t)$ and $\xi_0(t)$ forms a Hopf link.
\end{lemma}
\begin{proof}
	Assume $\xi(t)$ is a brake orbit. It is also a $z$-symmetric orbit if and only if $p_{r}(t_1)=0$. By Lemma  \ref{lemma:relationship pr r}, we get the unique root $r(0)=\frac{\om}{\sqrt{-2h}}$. So there exists one and only one $z$-symmetric brake orbit $\xi_{bz}(t)$ besides $\xi_0(t)$ within $\mathcal{M}$. Due to these two orbits linked together exactly once, they form a Hopf link. 
\end{proof}
\begin{remark}
	The orbit $\xi_{bz}(t)$ corresponds to a circular solution of \eqref{eq:spacial kepler Ham1} with angular momentum $\om$ about the $z$-axis.
\end{remark}

\section{The Kepler Problem with a symmetric perturbation}\label{sec:KPSP}

Now we will consider the following perturbed Kepler problem
$$ H_{\ep}(p, q) = \ey |p|^2 - \frac{1}{|q|}+\ep f(q, \ep). $$
with $f$ satisfying conditions \eqref{eq;symm-cond}. Rotational symmetry condition in \eqref{eq;symm-cond} implies the angular momentum with respect to the $z$-axis is still a first integral. Following the approach given in the previous section, by fixing it to be a non-zero constant $\om$, we get a reduced Hamiltonian system 
\begin{equation*}
(\dot{p}_r, \dot{p}_z, \dot{r}, \dot{z})^T = J \nabla H_{\om, \ep}(p_r, p_z, r, z).
\end{equation*}
with the reduced Hamiltonian $H_{\om, \ep}$ given as \eqref{eq:spacial kepler pertur Ham}. 

First, we establish properties analogous to those of the Kepler problem.  

\begin{proof}[Proof of Proposition \ref{prop;diffeomorphic to S3}] 
	By Proposition \ref{prop;diffeomorphic to S3 1}, when $ 2h\om^2 \in (-1, 0) $, $ \mathcal{M}(h, \om, 0) $ is a compact manifold diffeomorphic to $ \mathbb{S}^3 $, where  
	$$
	\min\left\{ r : (p_r, p_z, r, z) \in \mathcal{M}(h, \om, 0) \right\} = \frac{\om^2}{1 + \sqrt{1 + 2h\om^2}} > \frac{\om^2}{2}.  
	$$ 
	For $ \ep $ sufficiently small, $ \mathcal{M}(h, \om, \ep) $ remains close to $ \mathcal{M}(h, \om, 0) $.
	 
	As in the Kepler problem, the Hamiltonian $H_{\om,\ep}$, whose vector field $X_{H_{\om,\ep}}$ is defined by $ \omega(\ \cdot\ , X_{H_{\om,\ep}}) = \mathrm{d} H_{\om,\ep} $ with symplectic form $ \omega = \mathrm{d}p_r \wedge \mathrm{d}r + \mathrm{d}p_z \wedge \mathrm{d}z $, is a mechanical Hamiltonian. Consequently, the hypersurface $ \mathcal{M}(h, \om, \ep) $ is of contact type. Thus, there exists a contact form $ \lambda_{\ep} $ on $ \mathcal{M}(h, \om, \ep) $ satisfying $ \lambda_{\ep} \wedge \mathrm{d}\lambda_{\ep} \neq 0 $. Its Reeb vector field $ R_{\ep} $, determined by $ \mathrm{d}\lambda_{\ep}(R_{\ep}, \cdot) = 0 $ and $ \lambda_{\ep}(R_{\ep}) = 1 $, is parallel to $ X_{H_{\om,\ep}} $, with $ \mathrm{d}\lambda_{\ep} = \omega $. 
\end{proof}

\begin{remark}  
	While the contact form $ \lambda_{\ep} $ and Reeb vector field $ R_{\ep} $ differ from those of the Kepler problem, we simplify notation by writing $ \lambda $ and $ R $ in this section, as computations depend only on the symplectic form $ \omega $.  
\end{remark}

\begin{lemma}
\label{lem:xi-r-ep} There always exists a periodic solution $ \xi_{\ep}(t) $ entirely contained in the $ (p_r, r) $-plane.
\end{lemma}

\begin{proof}
Due to the $ z $-symmetry of the system, the subsystem restricted to the $ (p_r, r) $-plane is governed by the Hamiltonian  
$$
H_{\om,\ep}(p_r, r) = \frac{1}{2}p_r^2 + \frac{\om^2}{2r^2} - \frac{1}{r} + \ep f(r, 0, \ep),  
$$  
whose solution is $ \xi_{\ep}(t) $.
\end{proof}
 
\begin{lemma}\label{lemma:global serface of section}
	The sets  
	$$
	\begin{aligned}  
		\Sigma_{+,\ep} &= \left\{ (p_r, p_z, r, z) \in \mathcal{M}(h, \om, \ep) : z = 0, \, p_z \geq 0 \right\}, \\  
		\Sigma_{-,\ep} &= \left\{ (p_r, p_z, r, z) \in \mathcal{M}(h, \om, \ep) : z = 0, \, p_z \leq 0 \right\},  
	\end{aligned}  
	$$  
	are disk-like global surfaces of section for the reduced Hamiltonian flow \( \varphi_{\ep}^t \), with \( \xi_{\ep} \) as their common boundary.
\end{lemma}
\begin{proof}  
	The proof mirrors Lemma \ref{lemma:global serface of section 1}, leveraging the $ z $-symmetry and the fact that the system is a perturbation of the Kepler problem.  
\end{proof} 

Recent work by Cristofaro-Gardiner, Hryniewicz, Hutchings, and Liu \cite{CHHL23} established an algebraic relationship between the contact volume of a contact manifold and the action and Seifert rotation number of Reeb orbits when the manifold admits exactly two Reeb orbits:  
	\begin{thm}\label{thm:The condition of only two}  
		\cite[Theorem 1.2 and 1.5]{CHHL23} Let $ S $ be a 3-sphere with a contact form $ \lambda $ supporting exactly two Reeb orbits, $ \xi_1 $ and $ \xi_2 $. Let $ A_i $ and $ \text{Rot}_i $ denote the action and Seifert rotation number of $ \xi_i $ ($ i = 1, 2 $), and let $ \text{Vol}(S, \lambda) $ be the contact volume. Then  
		\begin{equation}\label{eq:Kepler equal}  
			\text{Vol}(S, \lambda) = \frac{A_1^2}{\text{Rot}_1} = \frac{A_2^2}{\text{Rot}_2},  
		\end{equation}  
		and both Reeb orbits are irrationally elliptic.  
	\end{thm} 
\begin{remark}
	 We say a Reeb orbit $\xi$ is elliptic if all eigenvalues of its monodromy matrix $\gamma$ lie on $\mathbb{U}$ and are semi-simple, where $\mathbb{U}$ denotes the unit circle in the complex plane. If, additionally, the angle $\vartheta$ is irrational for the eigenvalues $\sigma = e^{i\vartheta}$ and $\sigma^{-1} = e^{-i\vartheta}$, we say $\xi$ is irrationally elliptic. A Reeb orbit $\xi$ is hyperbolic if all eigenvalues of its monodromy matrix $\gamma$ are real and distinct from $\pm 1$. 
\end{remark}
By Lemma \ref{lemma:global serface of section}, $ \Sigma_{-,\ep} $ is a global surface of section. The first return time is defined as  
$$  
\hat{\tau}: \Sigma_{-,\ep} \setminus \xi_{\ep} \to \mathbb{R}^+, \quad z \mapsto \inf \left\{ t > 0 \, \big| \, \varphi_\ep^t(z) \in \Sigma_{-,\ep} \right\},  
$$  
and the first return map  
$$  
\psi: \Sigma_{-,\ep} \setminus \xi_{\ep} \to \Sigma_{-,\ep} \setminus \xi_{\ep}, \quad z \mapsto \varphi_\ep^{\hat{\tau}(z)}(z),  
$$  
is an area-preserving disk map where periodic points of $ \psi $ correspond to periodic orbits of $ \varphi_\ep^t $.

Combined with Franks' Theorem \cite{Franks92}: 
\begin{thm}\label{thm:disk map}  
	Let $ \mathbb{D} $ be an open disk and $ \mathfrak{F}: \mathbb{D} \to \mathbb{D} $ an area-preserving disk map. If $ \mathfrak{F} $ has two distinct periodic points, then it has infinitely many periodic points.  
\end{thm}  
If the algebraic relation in Theorem \ref{thm:The condition of only two} fails, the system admits infinitely many periodic orbits.  

We now analyze how the contact volume of \( \mathcal{M}(h, \om, \ep) \) and the action/Seifert rotation number of $ \xi_{\ep} $ differ from the Kepler problem. Let $ \mathcal{M}(h, \om, \ep) $ be abbreviated as $ \mathcal{M}_\ep $.

\subsection{Estimate the contact volume and action value} Following the definitions in \cite{CHHL23}, the contact volume of $\mathcal{M}_{\ep}$ and the action value of $\xi_{\ep}$ are  
\begin{equation}\label{eq;contect volume}  
	\text{Vol}(\mathcal{M}_{\ep}) = \int_{\mathcal{M}_{\ep}} \lambda \wedge \mathrm{d}\lambda = \int_{S_{\ep}} \mathrm{d}\lambda \wedge \mathrm{d}\lambda = 2 \int_{S_{\ep}} \mathrm{d}p_r \wedge \mathrm{d}r \wedge \mathrm{d}p_z \wedge \mathrm{d}z,  
\end{equation}  
and  
\begin{equation}\label{eq:action}  
	A(\xi_{\ep}(t)) = \int_{\xi_{\ep}(t)} \lambda = \int_{\Upsilon_{\ep}} \mathrm{d}\lambda = \int_{\Upsilon_{\ep}} \mathrm{d}p_r \wedge \mathrm{d}r,  
\end{equation}  
where $S_{\ep}=\{H_{\om,\ep}\leq h\}$ and $\Upsilon_{\ep}=S_{\ep}\cap\{p_z=0,z=0\}$. 

For brevity, let $\text{Vol}(\mathcal{M}_{\ep})$ and $A(\xi_{\ep}(t))$ be abbreviated as $\text{Vol}_{\ep}$ and $A_{\ep}$, respectively.
\begin{lemma}\label{lemma:the derivative of action and contact volume}
When $(\om,h)\in\Gamma_\ep$, let $r(\theta)=\om^{2}\slash(1+e\cos\theta)$, $e=\sqrt{1+2h\om^2}$, we have  
$$ \left. \frac{\mathrm{d}\text{Vol}_{\ep}}{\mathrm{d}\ep}\right| _{\ep=0}=-4\pi\tilde{V}(f), \; \text{ with } \tilde{V}(f)=\iint_{\mathcal{H}} f(r,z,0)\ \mathrm{d}r\mathrm{d}z;$$
$$ \left.\frac{\mathrm{d}A_{\ep}}{\mathrm{d}\ep}\right|_{\ep=0}=-\frac{2}{\om}\tilde{A}(f), \; \text{ with } \tilde{A}(f)=\int_{0}^{\pi}r^{2}(\theta)f(r(\theta),0,0)\mathrm{d}\theta.$$ 
\end{lemma}
The proof of Lemma \ref{lemma:the derivative of action and contact volume} relies on the following result:  
\begin{lemma}\label{lemma:The volume of D}
	Let $F(x,y,\ep)\in C^{\infty}(\mathbb{R}^{2}\times(-\ep_{0},\ep_{0}),\mathbb{R})$ and $D\subset\mathbb{R}^{2}$ be a closed domain, $\partial D$ is the boundary of $D$. Assume $\phi^{\ep}$ is a family of diffeomorphisms, $F(x,y,0)|_{\partial D}=0$, and $\left.\frac{\partial F}{\partial x}(x,y,0)\right|_{\partial D}\neq 0$, $\left.\frac{\partial F}{\partial y}(x,y,0)\right|_{\partial D}\neq 0$, then
	\begin{equation}
		\left. \frac{\mathrm{d}}{\mathrm{d}\ep}\right| _{\ep=0}\int_{\phi^{\ep}D}F(x,y,\ep)\mathrm{d}x\mathrm{d}y=\int_{D} \frac{\partial F}{\partial \ep}(x,y,0)\mathrm{d}x\mathrm{d}y.
	\end{equation}
\end{lemma}
\begin{proof}
	Assume $\phi^{\ep}(u,v)=(x(u,v,\ep),y(u,v,\ep))$. Then
	\begin{equation}
		\begin{aligned}
			&\left.\frac{\mathrm{d}}{\mathrm{d}\ep}\right|_{\ep=0}\int_{\phi^{\ep}D}F(x,y,\ep)\mathrm{d}x\mathrm{d}y\\
			=&\left. \frac{\mathrm{d}}{\mathrm{d}\ep}\right| _{\ep=0}\int_{D} (\phi^{\ep})^{*} F(x,y,0)\mathrm{d}x\wedge\mathrm{d}y+\int_{D} \frac{\partial F}{\partial \ep}(x,y,0)\mathrm{d}x\mathrm{d}y.
		\end{aligned}
	\end{equation}
	It suffices to show $\left. \frac{\mathrm{d}}{\mathrm{d}\ep}\right| _{\ep=0}\int_{D} (\phi^{\ep})^{*} F(x,y,0)\mathrm{d}x\wedge\mathrm{d}y=0$. Denote $\left. \frac{\partial x}{\partial \ep}\right|_{\ep=0}$ and $\left. \frac{\partial y}{\partial \ep}\right|_{\ep=0}$ by $s^{1}$ and $s^{2}$ respectively, then
	\begin{equation}
		\begin{aligned}
			&\left. \frac{\mathrm{d}}{\mathrm{d}\ep}\right| _{\ep=0}\int_{D} (\phi^{\ep})^{*} F(x,y,0)\mathrm{d}x\wedge\mathrm{d}y\\
			=&\left. \frac{\mathrm{d}}{\mathrm{d}\ep}\right| _{\ep=0}\int_{D}F(x(u,v,\ep),y(u,v,\ep),0)\left| \frac{\partial(x,y)}{\partial(u,v)}\right| \mathrm{d}u\mathrm{d}v\\
			=&\int_{D}\left( \frac{\partial F}{\partial x}s^{1}+\frac{\partial F}{\partial y}s^{2}\right) \left| \frac{\partial(x,y)}{\partial(u,v)}\right| +F\left( \frac{\partial s^{1}}{\partial x}+\frac{\partial s^{2}}{\partial y}\right)\left| \frac{\partial(x,y)}{\partial(u,v)}\right| \mathrm{d}u\mathrm{d}v\\
			=&\int_{D}\left( \frac{\partial F}{\partial x}s^{1}+F\frac{\partial s^{1}}{\partial x}\right) +\left( \frac{\partial F}{\partial y}s^{2}+F\frac{\partial s^{2}}{\partial y}\right) \mathrm{d}x\mathrm{d}y\\
			=&\int_{D}\frac{\partial (F\cdot s^{1})}{\partial x}+\frac{\partial (F\cdot s^{2})}{\partial y} \mathrm{d}x\mathrm{d}y\\
			=&\oint_{\partial D}-F\cdot s^{2}\mathrm{d}x+F\cdot s^{1}\mathrm{d}y.
		\end{aligned}
	\end{equation}
	Since $F(x,y,0)|_{\partial D}=0$ and $\frac{\partial F}{\partial x}, \frac{\partial F}{\partial y} \neq 0$ on $\partial D$, the terms $s^1$ and $s^2$ remain finite. Thus, the integral vanishes.
\end{proof}
\begin{proof}[Proof of Lemma \ref{lemma:the derivative of action and contact volume}]
	By \eqref{eq;contect volume}, we have
	\begin{equation}
		\text{Vol}_{\ep}=2\pi\iint_{\mathcal{H}_{\ep}}2h-\frac{\om^{2}}{r^{2}}+\frac{2}{\sqrt{r^{2}+z^{2}}}-2\ep f(r,z,\ep)\ \mathrm{d}r\mathrm{d}z.
	\end{equation}
	By Lemma \ref{lemma:The volume of D}, we get
	\begin{equation}
		\left. \frac{\mathrm{d}\text{Vol}_{\ep}}{\mathrm{d}\ep}\right| _{\ep=0}=-4\pi\iint_{\mathcal{H}} f(r,z,0)\ \mathrm{d}r\mathrm{d}z=-4\pi\tilde{V}(f).
	\end{equation}
	Here $\mathcal{H}_\ep$ is Hill's region of $\mathcal{M}_\ep$ and $\mathcal{H}$ is Hill's region of $\mathcal{M}$.
	
	By \eqref{eq:action}, then we get
	\begin{equation}
		\begin{aligned}
			\left.\frac{\mathrm{d}A_{\ep}}{\mathrm{d}\ep}\right|_{\ep=0}&=2\left.\frac{\mathrm{d}}{\mathrm{d}\ep}\right|_{\ep=0}\int_{r_{1,\ep}}^{r_{2,\ep}}\left(2h-\frac{\om^{2}}{r^{2}}+\frac{2}{r}-2\ep f(r,0,\ep)\right)^{1/2}\mathrm{d}r\\
			&=-2\int_{r_{1,0}}^{r_{2,0}}f(r,0,0)\left(2h-\frac{\om^{2}}{r^{2}}+\frac{2}{r}\right)^{-1/2}\mathrm{d}r\\
			&=-\frac{2}{\om}\int_{0}^{\pi}r^{2}(\theta)f(r(\theta),0,0)\mathrm{d}\theta=-\frac{2}{\om}\tilde{A}(f).
		\end{aligned}
	\end{equation} 
	Here we set $r_{1,\ep}$ and $r_{2,\ep}$ are the roots of the equation $\frac{\om^{2}}{2r^{2}}-\frac{1}{r}+\ep f(r,0,\ep)=h$.
\end{proof}



\subsection{Seifert rotation number of $\xi_{\ep}$} 
In this subsection, we estimate the Seifert rotation number of  $\xi_{\ep}$. Let $\hat{i}(\xi_{\ep})$ denote the mean (Maslov-type) index of the periodic orbit $\xi_{\ep}$. The rotation number of $\xi_{\ep}$ is defined as  
	$
	\overline{\text{Rot}}(\xi_{\ep}) = \frac{\hat{i}(\xi_{\ep})}{2}.  
	$  
Since the self-linking number of $\xi_{\ep}$ is $-1$, by Remark 4.4 of \cite{CHHL23}, the Seifert rotation number of $\xi_{\ep}$ is  
\begin{equation}\label{eq;rot_ep}
	\text{Rot}(\xi_{\ep})=\overline{\text{Rot}}(\xi_{\ep})-1=\frac{\hat{i}(\xi_{\ep})}{2}-1.
\end{equation} 

The definitions and properties of the Maslov-type index (also known as the Conley-Zehnder index) and the mean index are listed in Appendix, for more details, we refer the reader to \cite{Lon02}. We denote $\text{Rot}(\xi_{\ep})$ by $\text{Rot}_\ep$ below.

To calculate $\text{Rot}_\ep$, we analyze the linearized equation along $\xi_{\ep}(t)$  
\begin{equation}\label{eq:linearized}
\dot{y} = J D^2 H_{\om,\ep}(\xi_{\ep}(t)) y. 
\end{equation} 
Using $\theta$ as the time variable, let $T_{\ep}$ denote the period of $\xi_{\ep}(\theta)$. Recall that $r(\theta)=\om^{2}\slash(1+e\cos\theta)$, and let $r_{1,\ep}$ and $r_{2,\ep}$ be the roots of  
$
\frac{\om^2}{2r^2} - \frac{1}{r} + \ep f(r, 0, \ep) = h.  
$  
From $r^2 \dot{\theta} = \om$, the period $T_{\ep}$ is expressed as
\begin{equation}
	\begin{aligned}
		T_{\ep}&=2\int_{r_{1,\ep}}^{r_{2,\ep}}\frac{\om}{r^{2}}\cdot\left(2h-\frac{\om^{2}}{r^{2}}+\frac{2}{r}-2\ep f(r,0,\ep)\right)^{-1/2}\mathrm{d}r\\
		&=2\frac{\mathrm{d}}{\mathrm{d}h}\int_{r_{1,\ep}}^{r_{2,\ep}}\frac{\om}{r^{2}}\left(2h-\frac{\om^{2}}{r^{2}}+\frac{2}{r}-2\ep f(r,0,\ep)\right)^{1/2}\mathrm{d}r,
	\end{aligned}
\end{equation}
yielding  
\begin{equation}\label{eq;tilde Tf}
	\begin{aligned}
		\tilde{T}(f):=\left.\frac{\mathrm{d}T_{\ep}}{\mathrm{d}\ep}\right|_{\ep=0}&=-2\frac{\mathrm{d}}{\mathrm{d}h}\int_{r_{1,0}}^{r_{2,0}}f(r,0,0)\frac{\om}{r^{2}}\cdot\left(2h-\frac{\om^{2}}{r^{2}}+\frac{2}{r}\right)^{-1/2}\mathrm{d}r\\
		&=-2\frac{\mathrm{d}}{\mathrm{d}h}\int_{0}^{\pi}f(r(\theta),0,0)\mathrm{d}\theta\\
		&=\frac{2}{e}\int_{0}^{\pi} r^{2}(\theta)\frac{\partial f}{\partial r}(r(\theta),0,0)\cos\theta\mathrm{d}\theta.
	\end{aligned}
\end{equation}
In spherical coordinates $(\rho, \tht, \varphi)$ with momenta $(p_{\rho},p_{\tht},p_{\varphi})=(\dot{\rho}, \rho^2\cos^2\varphi \dot{\tht},\rho^2\dot{\varphi})$, the reduced Hamiltonian is 
\begin{equation}
	H_{\om,\ep}(p_{\rho},p_{\varphi},\rho ,\varphi)=\frac{1}{2}(p_{\rho}^{2}+\frac{p_{\varphi}^{2}}{\rho^{2}})+\frac{\om^{2}}{2\rho^{2}\cos^{2}\varphi}-\frac{1}{\rho}+\ep f(\rho,\varphi,\ep).
\end{equation} 

Applying the time-scaling transformation $\frac{\mathrm{d}\tau}{\mathrm{d}t}=\frac{\om}{\rho^2}$, we define
\begin{equation}
	K_{\om,\ep}=\frac{\rho^{2}}{\om}(H_{\om,\ep}-h)=\frac{1}{\om}\left(\frac{1}{2}\rho^{2}p_{\rho}^{2}+\frac{1}{2}{p_{\varphi}^{2}}-h\rho^{2}+\frac{\om^{2}}{2\cos^{2}\varphi}-\rho+\ep g(\rho,\varphi,\ep)\right),
\end{equation}
where $g(\rho,\varphi,\ep)=\rho^{2} f(\rho,\varphi,\ep)$. Since  
\begin{equation}
	\nabla K_{\om,\ep}=\frac{\rho^2}{\om}\nabla H_{\om,\ep}+\nabla(\rho^2\slash\om)(H_{\om,\ep}-h),
\end{equation}
the solutions of $\dot{x}=J\nabla H_{\om,\ep}(x)$ coincide with those of $\frac{\mathrm{d}x}{\mathrm{d}\tau}=J\nabla K_{\om,\ep}(x)$ on the energy surface $K_{\om,\ep}^{-1}(0)=H_{\om,\ep}^{-1}(h)$. For $\xi_{\ep}$, as a solution of  $\frac{\mathrm{d}x}{\mathrm{d}\tau}=J\nabla K_{\om,\ep}(x)$, the period in $\tau$ is $T_{\ep}$.
The gradient and Hessian of \(K_{\om,\ep}\) are 
$$
\nabla K_{\om,\ep} = \frac{1}{\om} \left( \rho^2 p_\rho, \, p_\varphi, \, \rho p_\rho^2 - 2h \rho - 1 + \ep \frac{\partial g}{\partial \rho}, \, \frac{\om^2 \sin \varphi}{\cos^3 \varphi} + \ep \frac{\partial g}{\partial \varphi} \right),  
$$ 
and  
$$
D^2 K_{\om,\ep} = \frac{1}{\om} \begin{pmatrix}  
	\rho^2 & 0 & 2 \rho p_\rho & 0 \\  
	0 & 1 & 0 & 0 \\  
	2 \rho p_\rho & 0 & p_\rho^2 - 2h + \ep \frac{\partial^2 g}{\partial \rho^2} & \ep \frac{\partial^2 g}{\partial \rho \partial \varphi} \\  
	0 & 0 & \ep \frac{\partial^2 g}{\partial \rho \partial \varphi} & \frac{\om^2 (1 + 2 \sin^2 \varphi)}{\cos^4 \varphi} + \ep \frac{\partial^2 g}{\partial \varphi^2}  
\end{pmatrix}.  
$$
For $\xi_{\ep}(\tau)$ restricted to the $\varphi$-axis ($\varphi(\xi_{\ep}) = 0$), we have
$$
D^2 K_{\om,\ep}(\xi_{\ep}) = \frac{1}{\om} \begin{pmatrix}  
	\rho^2 & 0 & 2 \rho p_\rho & 0 \\  
	0 & 1 & 0 & 0 \\  
	2 \rho p_\rho & 0 & p_\rho^2 - 2h + \ep \frac{\partial^2 g}{\partial \rho^2} & 0 \\  
	0 & 0 & 0 & \om^2 + \ep \frac{\partial^2 g}{\partial \varphi^2}  
\end{pmatrix}.  
$$ 
The linearized equation $\frac{\mathrm{d}y}{\mathrm{d}\tau}=JD^{2}K_{\om,\ep}(\xi_{\ep}(\tau))y$ decouples into two independent subsystems  
\begin{equation}\label{eq;linear 1}
	\frac{\mathrm{d}\eta_{1}}{\mathrm{d}\tau}=\frac{1}{\om}J\left({\begin{array}{cc}
			\rho^{2} & 2\rho p_{\rho}\\
			2\rho p_{\rho} & p_{\rho}^{2}-2h+\ep\frac{\partial^{2}g}{\partial \rho^{2}}
	\end{array}}\right)\eta_{1},
\end{equation}
and
\begin{equation}\label{eq;linear 2}
	\frac{\mathrm{d}\eta_{2}}{\mathrm{d}\tau}=\frac{1}{\om}J\left({\begin{array}{cc}
			1 & 0\\
			0 & \om^{2}+\ep\frac{\partial^{2}g}{\partial \varphi^{2}}
	\end{array}}\right)\eta_{2}.
\end{equation}  
\begin{lemma}\label{lemma;only deponds on 2}
		Let $\gamma_\ep(\tau)$ be the fundamental solution of \eqref{eq;linear 2}, then $\text{Rot}_\ep=\hat{i}(\gamma_\ep(\tau))/2$.
	\end{lemma}
\begin{proof}
	Since $\xi_{\ep}(\tau)=(p_{\rho,\ep}(\tau),0,\rho_\ep(\tau),0)$ satisfies  
	$\frac{\mathrm{d}\xi_{\ep}(\tau)}{\mathrm{d}\tau}=J\nabla K_{\om,\ep}(\xi_{\ep}(\tau))$,
	differentiating both sides with respect to $\tau$ yields  
	$$
	\frac{\mathrm{d}^2\xi_{\ep}(\tau)}{\mathrm{d}\tau^2} = J D^2 K_{\om,\ep}(\xi_{\ep}(\tau)) \frac{\mathrm{d}\xi_{\ep}(\tau)}{\mathrm{d}\tau}.
	$$  
	Thus, $c \cdot \left( \frac{\mathrm{d}p_{\rho,\ep}(\tau)}{\mathrm{d}\tau}, \frac{\mathrm{d}\rho_\ep(\tau)}{\mathrm{d}\tau} \right)$ is a periodic solution of \eqref{eq;linear 1}. Because $\frac{\mathrm{d}\rho_\ep(\tau)}{\mathrm{d}\tau}$ vanishes twice over a period $T_\ep$, equation \eqref{eq;linear 1} contributes $2$ to the mean index. By \eqref{eq;rot_ep}, the result follows.
\end{proof}

For a $2\times2$ symplectic path from mechanical systems, the analysis is simpler than in the general case. Let $\sigma_{\ep}(\tau)=\rho(\tau)e^{i\vartheta(\tau)}$ and $\sigma_{\ep}^{-1}(\tau)$ denote the eigenvalues of $\gamma_{\ep}(\tau)$, where $\vartheta(t)$ is a continuous non-decreasing function. The mean index is then given by  
	$
	\hat{i}(\gamma_{\ep}(\tau)) = \frac{\vartheta(T_{\ep}) - \vartheta(0)}{\pi}.  
	$  
However, directly calculating the derivative of rotation number $\text{Rot}_\ep$ remains challenging. To address this, we utilize the trace formula introduced by the first author and collaborators \cite{HOW15}. This framework was originally developed to analyze the linear stability of the Lagrange elliptic relative equilibria. Below, we briefly summarize relevant results for our purposes. Let $T>0$, $I_{2n}$ denote the $2n\times2n$ identity matrix, and  
$  
J = \begin{pmatrix} 0 & -I_n \\ I_n & 0 \end{pmatrix}  
$  
be the standard symplectic matrix. Let $\mathcal{B}(k)$ denote the space of $k\times k$ symmetric matrices, and let $B(t),D(t)\in C([0,T],\mathcal{B}(2n))$. For the anti-periodic linear Hamiltonian system  
\begin{equation}\label{eq: linear Hamiltonian system}
	\dot{z}(t)=J(B(t)+\lambda D(t))z(t),\quad z(0)=-z(T),
\end{equation}
the operator $-J\frac{\mathrm{d}}{\mathrm{d}t}$ is densely defined in the Hilbert space $E=L^{2}([0,T],\mathbb{C}^{2n})$ with domain 
\begin{equation*}
	\mathfrak{D}_{-1} = \left\{ z(t) \in W^{1,2}([0, T], \mathbb{C}^{2n}) \, \big| \, z(0) = -z(T) \right\}. 
\end{equation*}
Assuming $-J\frac{\mathrm{d}}{\mathrm{d}t}-B(t)$ is invertible on  $\mathfrak{D}_{-1}$, the operator $\mathcal{F} = D \left( -J\frac{\mathrm{d}}{\mathrm{d}t} - B \right)^{-1}$ is of trace class.
\begin{lemma}\label{lemma: trace formular1}
	\cite[Theorem 3.3]{HOW15} Let $\gamma_{\lambda}(t)$ is the fundamental solution of \eqref{eq: linear Hamiltonian system}, then
	\begin{equation*}
		\det(id-\lambda\mathcal{F})= \det(\gamma_{\lambda}(T)+I_{2n})\cdot\det(\gamma_{0}(T)+I_{2n})^{-1}.
	\end{equation*}
\end{lemma}
\begin{lemma}\label{lemma;det}
	\cite[Theorem 2.2]{HOW15} Under the above assumption, for $\lambda$ small enough, we have
	\begin{equation*}
		\det(id-\lambda\mathcal{F})=\exp\left\{\sum_{n=1}^{\infty}b_n\lambda^n\right\},
	\end{equation*}
	where $b_n=-\frac{1}{n}Tr(\mathcal{F}^n)$.
\end{lemma}
\begin{lemma}\label{lemma: trace formular2}
	\cite[Corollary 1.3]{HOW15} If $\gamma_{0}(T)=I_{2n}$, then 
	\begin{equation*}
		Tr(\mathcal{F})=-\frac{1}{2}Tr(\mathcal{D})=0\quad\text{ and   }\quad Tr(\mathcal{F}^{2})=-\frac{1}{4}Tr(\mathcal{D}^{2}),
	\end{equation*}
	where $\mathcal{D}=J\int_{0}^{T}\gamma_{0}^{T}(s)D(s)\gamma_{0}(s)\mathrm{d}s$.
\end{lemma}

Using the above results, we get the following theorem. 
\begin{thm}\label{thm:the derivative of rotation number}
	Define $D(f)$ and $E(f)$ as below, where $\tilde{T}(f)$ is introduced in \eqref{eq;tilde Tf}, 
	$$ E(f)=\int_{0}^{\pi}r^{2}(\theta)\frac{\partial^{2}f}{\partial\varphi^{2}}(r(\theta),0,0)\mathrm{d}\theta+\om^{2}\tilde{T}(f);$$
	$$ D(f)=E^2(f)-\left(\int_{0}^{\pi}r^{2}(\theta)\frac{\partial^{2}f}{\partial\varphi^{2}}(r(\theta),0,0)\cos2\theta\mathrm{d}\theta\right)^{2}.$$
	We have
	$$
	\frac{\mathrm{d}\text{Rot}_{\ep}}{\mathrm{d}\ep} |_{\ep=0} = \begin{cases}
		0, & \text{ if } D(f)<0; \\ 
		\frac{\text{sign}(E(f))}{2\pi\om^{2}}\sqrt{D(f)}, & \text{ if } D(f)\geq0,
	\end{cases} 
	$$where $\text{sign}(a)=a/|a|$ when $a\neq0$ and $\text{sign}(0)=0$.
\end{thm}

\begin{proof}
	To estimate the rotation number $\text{Rot}_{\ep}$ via the trace formula, we reformulate equation \eqref{eq;linear 2} as an anti-periodic boundary value problem. First, applying the symplectic transformation $\tilde{\eta}_{2} = \text{diag}(1/\sqrt{\om}, \sqrt{\om}) \eta_{2}$, we obtain
	\begin{equation*}
		\frac{\mathrm{d}\tilde{\eta}_2}{\mathrm{d}\tau}=J\left({\begin{array}{cc}
				1 & 0\\
				0 & 1+\frac{\ep}{\om^{2}}\cdot\frac{\partial^{2}g}{\partial \varphi^{2}}
		\end{array}}\right)\tilde{\eta}_{2},
	\end{equation*}
	Reparameterizing $\tau$ to $s$ via $\frac{\mathrm{d}\tau}{\mathrm{d}s}=\frac{T_{\ep}}{2\pi}$ (denoting $'=\frac{\mathrm{d}}{\mathrm{d}s}$), the system becomes 
	\begin{equation}\label{eq:linear equation 0}
		\tilde{\eta}_{2}'=\frac{T_{\ep}}{2\pi}J\left({\begin{array}{cc}
				1 & 0\\
				0 & 1+\frac{\ep}{\om^{2}}\cdot\frac{\partial^{2}g}{\partial \varphi^{2}}
		\end{array}}\right)\tilde{\eta}_{2},
	\end{equation}
	where the period of $\xi_{\ep}$ in $s$ is $2\pi$. The fundamental solution of \eqref{eq:linear equation 0} is a reparameterization of $\gamma_{\ep}(\tau)$, denoted $\gamma_{\ep}(s)$ with $s$ as the time variable.
	
	Consider the anti-periodic boundary value problem 
	\begin{equation}\label{eq:linear equation}
		\left\{\begin{aligned}
			&\tilde{\eta}'_2=J(I_2+\ep D_{\ep})\tilde{\eta}_2,\\
			&\tilde{\eta}_2(0)=-\tilde{\eta}_2(2\pi).
		\end{aligned}\right.
	\end{equation}
	where $D_{\ep}=\text{diag}\left(\frac{\tilde{T}(f)}{2\pi}+O(\ep),\, \frac{\tilde{T}(f)}{2\pi}+\frac{T_{\ep}}{2\om^{2}\pi}\cdot\frac{\partial^{2}g}{\partial \varphi^{2}}+O(\ep)\right)$ and $\tilde{T}(f)=\left.\frac{\mathrm{d}T_{\ep}}{\mathrm{d}\ep}\right|_{\ep=0}$. Since the operator $-J\frac{\mathrm{d}}{\mathrm{d}s}-I_2$ is invertible in $\mathfrak{D}_{-1}$, Lemma \ref{lemma: trace formular1} implies that for $\mathcal{F}_\ep=D_{\ep}\left(-J\frac{\mathrm{d}}{\mathrm{d}s}-I_2\right)^{-1}$, 
	\begin{equation*}
		\det(id-\ep\mathcal{F}_\ep)=\det(\gamma_{\ep}(2\pi)+I_2)\cdot\det(\gamma_{0}(2\pi)+I_2)^{-1}.
	\end{equation*}
	Direct computation shows $\gamma_0(s) = \begin{pmatrix} \cos s & -\sin s \\ \sin s & \cos s \end{pmatrix}$ and $\gamma_{0}(2\pi)=I_2$. Let $\sigma_{\ep}$ and $\sigma^{-1}_{\ep}$ denote the eigenvalues of $\gamma_{\ep}(2\pi)$. By Lemmas \ref{lemma;det} and \ref{lemma: trace formular2},
	\begin{equation*}
		1 - \frac{1}{2} Tr(\mathcal{F}_{\ep}^2) \ep^2 + o(\ep^2) = \frac{2 + \sigma_{\ep} + \sigma_{\ep}^{-1}}{4}.
	\end{equation*}
	Differentiating twice with respect to $\ep$ at $\ep=0$, we have
	\begin{equation*}
		-Tr(\mathcal{F}_0^2)= (\sigma_{0}')^{2}/2.
	\end{equation*}
	Using Lemma \ref{lemma: trace formular2} and the fact that $r(\theta)=\om^2/(1+e\cos\theta)$ is an even $2\pi$-periodic function, we compute
	\begin{equation*}
		-\left(\frac{1}{\om^{2}}\int_{0}^{\pi}r^{2}(\theta)\frac{\partial^{2}f}{\partial\varphi^{2}}(r(\theta),0,0)\mathrm{d}\theta+\tilde{T}(f)\right)^{2}+\frac{1}{\om^{4}}\left(\int_{0}^{\pi}r^{2}(\theta)\frac{\partial^{2}f}{\partial\varphi^{2}}(r(\theta),0,0)\cos2\theta\mathrm{d}\theta\right)^{2}=(\sigma_{0}')^{2}.
	\end{equation*}
    i.e.,
    \begin{equation}
    	-D(f)\slash \om^4=(\sigma_{0}')^{2}.
    \end{equation}

	Case 1: $D(f)<0$. Then $\sigma_{0}'$ is real, all eigenvalues of $\gamma_{\ep}(2\pi)$ are real, and $\text{Rot}_{\ep}=1=\text{Rot}_{0}$  with $\left.\frac{\mathrm{d}\text{Rot}_{\ep}}{\mathrm{d}\ep}\right|_{\ep=0}=0$.
	
	Case 2: $D(f)=0$. Then $\sigma_{0}'=0$, and the eigenvalue variation of $\gamma_{\ep}(2\pi)$ is of higher order in $\ep$, so $\left.\frac{\mathrm{d}\text{Rot}_{\ep}}{\mathrm{d}\ep}\right|_{\ep=0}=0$.
	
	Case 3: $D(f)>0$. Then $\sigma_{0}'$ is purely imaginary. Writing $\sigma_{\ep}=\exp(ic_{\ep})$, we have $(\sigma_{0}')^{2}=-(c_{0}')^{2}$, giving $\left.\frac{\mathrm{d}\text{Rot}_{\ep}}{\mathrm{d}\ep}\right|_{\ep=0}=\pm\frac{|c_{0}'|}{2\pi}=\pm\frac{1}{2\pi\om^{2}}\sqrt{D(f)}$. To determine the sign, rewrite \eqref{eq:linear equation} as 
	\begin{equation}\label{eq:linear equation2}
		\left\{\begin{aligned}
			&\tilde{\eta}_2'=J\left(I_2+\ep\frac{E(f)}{2\pi\om^{2}}I_2+\ep \left(D_{\ep}-\frac{E(f)}{2\pi\om^{2}}I_2\right)\right)\tilde{\eta}_2,\\
			&\tilde{\eta}_2(0)=-\tilde{\eta}_2(2\pi).
		\end{aligned}\right.
	\end{equation}
	and analyze the modified system 
	\begin{equation}\label{eq:linear equation3}
		\left\{\begin{aligned}
			&\tilde{\eta}_2'=J\left(I_2+\ep \left(D_{\ep}-\frac{E(f)}{2\pi\om^{2}}I_2\right)\right)\tilde{\eta}_2,\\
			&\tilde{\eta}_2(0)=-\tilde{\eta}_2(2\pi).
		\end{aligned}\right.
	\end{equation}
	Repeating the eigenvalue analysis for \eqref{eq:linear equation3}, the monodromy matrix eigenvalues are real. By Sturm’s Comparison Theorem: for $E(f)>0$, $\left.\frac{\mathrm{d}\text{Rot}_{\ep}}{\mathrm{d}\ep}\right|_{\ep=0}>0$; for $E(f)<0$, $\left.\frac{\mathrm{d}\text{Rot}_{\ep}}{\mathrm{d}\ep}\right|_{\ep=0}<0$.
\end{proof}
\begin{remark}
	Although $D(f)=D(-f)$, the derivative of $\text{Rot}_\ep$ is opposite in sign.
\end{remark}
The above result implies the following corollary. 

\begin{coro}\label{coro:linear stable}
	$\xi_{\ep}(t)$ is elliptic when $D(f)>0$ and hyperbolic when $D(f)<0$.
\end{coro}
\begin{remark}
	We say $\xi_{\ep}(t)$ is linearly stable, if $\xi_{\ep}(t)$ is elliptic, and linearly unstable, if $\xi_{\ep}(t)$ is hyperbolic.
\end{remark}
\subsection{The existence of infinitely many periodic orbits and a unique brake $z$-symmetric orbit which forms a Hopf link with $\xi_{\ep}$.}
Now we will show the perturbed system has infinitely many periodic solutions under certain conditions based on Theorem \ref{thm:The condition of only two} and Theorem \ref{thm:disk map}. 

\begin{thm}\label{thm:Exist infinity many}
	For sufficiently small $\ep$ and parameters $(\om,h) \in \Gamma_\ep$, let $C(e)=\frac{1}{\sqrt{1-e^{2}}}-1$, the reduced Hamiltonian system \eqref{eq;Ham-sys-pert} admits infinitely many periodic orbits within $\M_\ep$, if one of the following two conditions holds
	\begin{enumerate}
		\item[(i)] $D(f)<0$;
		\item[(ii)] $D(f)\geq0$ and $\tilde{V}(f)\neq2C(e)\tilde{A}(f)+C^{2}(e)\text{sign}(E(f))\sqrt{D(f)}\slash2$,
	\end{enumerate}
	where $D(f)$, $E(f)$ are defined in Theorem \ref{thm:the derivative of rotation number}, and $\tilde{V}(f)$, $\tilde{A}(f)$ are defined in Lemma \ref{lemma:the derivative of action and contact volume}.
\end{thm}
\begin{proof}
First, from Lemma \ref{lemma:kepler equal} and the fact that $\text{Rot}_0 = 1$, we conclude that the Kepler problem satisfies the algebraic relations stated in Theorem \ref{thm:The condition of only two}. For the perturbed system \eqref{eq;Ham-sys-pert}, by Theorems \ref{thm:The condition of only two} and \ref{thm:disk map}, if the derivatives of both sides with respect to $\ep$ are unequal, i.e. 
$$
\left.\frac{\mathrm{d}\text{Vol}_{\ep}}{\mathrm{d}\ep}\right|_{\ep=0} \neq 2A(\xi_{0}(t))\left.\frac{\mathrm{d}A_{\ep}}{\mathrm{d}\ep}\right|_{\ep=0} - A^2(\xi_{0}(t))\left.\frac{\mathrm{d}\text{Rot}_{\ep}}{\mathrm{d}\ep}\right|_{\ep=0},
$$ 
then infinitely many periodic orbits exist for sufficiently small $\ep$. By Theorem \ref{thm:the derivative of rotation number}, Lemma \ref{lemma:the derivative of action and contact volume}, and equation \eqref{eq:action of Kepler}, this inequality is equivalent to condition (ii). Furthermore, by Corollary \ref{coro:linear stable}, Theorem \ref{thm:The condition of only two}, and Theorem \ref{thm:disk map}, condition (i) implies that $\xi_{\ep}$ is hyperbolic and that infinitely many periodic orbits exist.
\end{proof}
\begin{remark}
	From previous definition, $\tilde{V}(f)$, $\tilde{A}(f)$ and $\sqrt{D(f)}$ are the scaled derivatives of $\text{Vol}_\ep$, $A_\ep$ and $\text{Rot}_\ep$ respectively. 
\end{remark}
While one can obtain infinitely many periodic orbits using the above theorem, it requires some additional conditions besides the perturbation being small enough. Meanwhile, our next result shows the existence of at least one periodic orbit when the symmetric perturbation is small enough without any further assumption.

\begin{proof}[Proof of Theorem \ref{thm:z-sym exist}]

We apply the Implicit Function Theorem (IFT) to prove this theorem. First, we analytically construct a one-to-one correspondence between the $ r $-coordinates associated with $\mathcal{B}_{+,\ep}$ and $\mathcal{B}_{+,0}$, where  
$$
\mathcal{B}_{+,\ep} = \left\{ \frac{\om^2}{2r^2} - \frac{1}{\sqrt{r^2 + z^2}} + \ep f(r, z, \ep) = h \right\}
$$  
is the upper boundary of the Hill’s region $\mathcal{H}_\ep$, while $\mathcal{B}_{+,0}$ denotes the upper boundary of the Hill’s region $\mathcal{H}$ for the unperturbed Kepler problem. Among possible approaches to establish this correspondence, we parametrize points on $\mathcal{B}_{+,\ep}$ and $\mathcal{B}_{+,0}$ by constructing rays from $(\om^2, 0)$ that intersect these boundaries at $(R(r, \ep), Z(r, \ep))$ and $(r, z(r))$, respectively.  

To apply the IFT, let $\varphi^t_\ep( p_{r_0}, p_{z_0}, r_0, z_0)$ denote the flow of the system in $(\ref{eq:spacial kepler pertur Ham})$, where $(p_{r_0}, p_{z_0}, r_0, z_0)$ are initial conditions and $\ep$ is the perturbation parameter.  

For every $(r, z(r)) \in \mathcal{B}_{+,0}$, we have $R(r, 0) = r$ and $Z(r, 0) = z(r)$. Define the function  
$$
f_1(t, r, \ep) = \Pi_z\varphi^{t}_\ep( 0, 0, R(r, \ep), Z(r, \ep)).
$$ 
When $\ep = 0$, Lemma \ref{lemma:global serface of section 1} ensures that $\Sigma_{-}$ is a global surface of section. Hence, for fixed $r$, there exists a smallest $T(r) > 0$ such that  
$$
\varphi_0^{T(r)}( 0, 0, r, z(r)) \in \Sigma_{-}^\circ.
$$  
This implies:  
$$
\begin{aligned}
	&f_1(T(r), r, 0) = \Pi_z\varphi_0^{T(r)}( 0, 0, r, z(r)) = 0, \\
	&\partial_t f_1(T(r), r, 0) = \Pi_{p_z}\varphi_0^{T(r)}( 0, 0, r, z(r)) < 0.
\end{aligned}
$$  
By the IFT, there exists a unique $\tilde{T}(r, \ep)$ near $( T(r), r, 0)$ such that $\tilde{T}(r, 0) = T(r)$ and  
$$
\Pi_z\varphi_\ep^{\tilde{T}(r, \ep)}( 0, 0, R(r, \ep), Z(r, \ep)) = 0.
$$  

Next, consider the function  
$$
f_2(r, \ep) = \Pi_{p_r}\varphi_\ep^{\tilde{T}(r, \ep)}( 0, 0, R(r, \ep), Z(r, \ep)).
$$  
For $\ep = 0$, let $r_1 = \frac{\om}{\sqrt{-2h}}$. By Lemmas \ref{lemma:relationship pr r} and \ref{lemma;the existence of bz for Kepler},  
$$
\begin{aligned}
	&f_2(r_1, 0) = \Pi_{p_r}\varphi_0^{T(r_1)}( 0, 0, r_1, z(r_1)) = \frac{\om^2 + 2hr_1^2}{2\om r_1} = 0,\\
	&\partial_r f_2(r_1, 0) = \frac{2h}{\om} \neq 0.
\end{aligned}
$$
Applying the IFT again, there exists a unique $\tilde{r}(\ep)$ near $(r_1,0)$ such that $\tilde{r}(0) = r_1$ and  
\[
\Pi_{p_r}\varphi_\ep^{\tilde{T}(\tilde{r}(\ep),\ep)}( 0, 0, R(\tilde{r}(\ep), \ep), Z(\tilde{r}(\ep), \ep)) = 0.
\]  
This guarantees the existence of a unique $z$-symmetric brake orbit $\xi_{bz, \ep}(t)$ forming a Hopf link with $\xi_{\ep}(t)$.
\end{proof}
\begin{remark}
	There may even be infinitely many $z$-symmetric brake orbits, but only $\xi_{bz, \ep}(t)$ forms a Hopf link with $\xi_{\ep}(t)$.
\end{remark}

\section{The ellipsoid problem}\label{sec:ellipsoid}
In this section, we apply results from the previous section to the ellipsoid problem. Recall that when the eccentricity $\fe$ is small enough, the potential function can be written as \eqref{eq;ell-U}. If we further assume $GM=1$, then 
\begin{equation}
	\begin{aligned}			  
		U(\rho,\varphi) &=-\frac{1}{\rho}+\frac{R^{2}}{5}\fe\frac{3\sin^{2}\varphi-1}{\rho^{3}}+O(\fe^{2})\\
		&=-\frac{1}{\rho}+\frac{3\sin^{2}\varphi-1}{\rho^{3}}\ep+O(\ep^{2}).
	\end{aligned}
\end{equation}
where $\ep=\frac{R^{2}}{5}\fe$ as a small parameter.


Theorem \ref{thm:axial symmetry ellipsoid potential} follows from the next result. 
\begin{thm}
	For sufficiently small $\ep$, when $(\om,h) \in \Gamma_\ep$, 
	\begin{enumerate}
		\item[(a)] the ellipsoid problem has a solution $\xi_{\ep}(t)$ that lies on the $r$-axis and it is linearly stable. 
		\item[(b)] there exists one and only one $z$-symmetric brake orbit $\xi_{bz,\ep}(t)$ in $\M_\ep$ which forms a Hopf link with $\xi_{\ep}(t)$.
		\item[(c)] the ellipsoid problem admits infinitely many relative periodic solutions within $\M_\ep$.
	\end{enumerate}
\end{thm}
 
\begin{proof} 
	Interpreting this system as a Kepler problem with a symmetric perturbation, we express the perturbation term as
	\begin{equation*}
		f(r,z,\ep)=\frac{2z^{2}-r^{2}}{(r^{2}+z^{2})^{5/2}}+O(\ep)
		=\frac{3\sin^{2}\varphi-1}{\rho^{3}}+O(\ep).
	\end{equation*}

    The existence of $\xi_{\ep}$ follows from Lemma \ref{lem:xi-r-ep}. Theorem \ref{thm:z-sym exist} implies property (b). To establish property (c), we apply Theorem \ref{thm:Exist infinity many} through the following calculations. 
    
    First, determine the sign of $D(f)$ by computing 
    \begin{equation*}
    	\tilde{T}(f) = \frac{2}{e}\int_{0}^{\pi} \frac{3}{r^{2}(\theta)}\cos\theta\mathrm{d}\theta = \frac{6\pi}{\om^{4}},\quad E(f)=\int_{0}^{\pi}\frac{6}{r(\theta)}\mathrm{d}\theta + \frac{6\pi}{\om^{2}}=\frac{12\pi}{\om^{2}},
    \end{equation*}
we obtain 
\begin{equation*}
	D(f) = \left(\frac{12\pi}{\om^{2}}\right)^{2} - \left(\int_{0}^{\pi}\frac{6}{r(\theta)}\cos2\theta\mathrm{d}\theta\right)^{2} 
	= \left(\frac{12\pi}{\om^{2}}\right)^{2} > 0.
\end{equation*}
By Corollary \ref{coro:linear stable}, $\xi_{\ep}(t)$ is linearly stable and this proves property (a). Then we need to verify condition (ii) in Theorem \ref{thm:Exist infinity many}.

    We compute $\tilde{V}(f)$ using polar coordinates with the substitution $\sin\varphi = e\sin\theta$ where $e = \sqrt{1+2h\om^{2}}$,
    \begin{equation}
    	\begin{aligned}
    		\tilde{V}(f) &= \int_{\mathcal{H}}f(r,z,0)\mathrm{d}r\mathrm{d}z 
    		= \int_{\mathcal{H}}\frac{3\sin^{2}\varphi-1}{\rho^{2}}\mathrm{d}\rho\mathrm{d}\varphi \\
    		&= \frac{2}{\om^2}\int_{\varphi^-}^{\varphi^+}(3\sin^2\varphi-1)\cos^2\varphi\left(1+\frac{2h\om^2}{\cos^2\varphi}\right)^{1/2}\mathrm{d}\varphi \\
    		&= \frac{2e^2}{\om^2}\int_{-\pi/2}^{\pi/2}(3e^2\sin^2\theta-1)\cos^2\theta\mathrm{d}\theta \\
    		&= -\frac{\pi}{4\om^{2}}e^{2}(4-3e^{2}).
    	\end{aligned}
    \end{equation}
Then we study the right-hand side,
\begin{equation*}	
	2C(e)\tilde{A}(f) + \frac{1}{2}C^{2}(e)\sqrt{D(f)} 
	= \frac{2\pi C(e)}{\om^{2}}(3C(e)-1),
\end{equation*}
with $\tilde{A}(f) = -\int_{0}^{\pi}\frac{1}{r(\theta)}\mathrm{d}\theta = -\frac{\pi}{\om^{2}}$.
	Crucially, the inequality
	\begin{equation}
		-\frac{\pi}{4\om^{2}}e^{2}(4-3e^{2}) < \frac{2\pi C(e)}{\om^{2}}(3C(e)-1) \quad \forall e\in(0,1)
	\end{equation}
	implies condition (ii) holds.
\end{proof}

\section{The n-pyramidal problem}\label{sec:pyramidal}

In this section, we apply the preceding results to the $n$-pyramidal problem. Recall that the fixed axis is the $z$-axis and the masses of the $1+n$ bodies are $m_{0},m_1,\cdots,m_n$ and $m_j=m$ for $j=1,2,\cdots,n$. Let $q_{1}=(x,y,z)\in\mathbb{R}^{3}$ be the position of $m_1$, since the mass center is fixed at the origin, the position of $m_0$ is $q_{0}=(0,0,-n\ep z)$, where $\ep=m\slash m_0$ is the mass ratio. Correspondingly, the velocity of $m_1$ is $\dot{q}_{1}=(\dot{x},\dot{y},\dot{z})$ and that of $m_0$ is $\dot{q}_0=(0,0,-n\ep\dot{z})$. Then the system can be described by the motion of $m_1$ and the Lagrangian is
\begin{equation}
	\begin{aligned}
		\hat{L}(q,\dot{q})&=\frac{nm}{2}(\dot{x}^{2}+\dot{y}^{2}+(1+n\ep)\dot{z}^2)\\
		&+\frac{1}{2}m^{2}\left(\frac{nM(n)}{\sqrt{x^{2}+y^{2}}}+\frac{2n\ep^{-1}}{\sqrt{x^{2}+y^{2}+(1+n\ep)^{2}z^{2}}}\right),
	\end{aligned}
\end{equation}
where 
\begin{equation}
	M(n)=\frac{1}{2}\sum\limits_{i=1}\limits^{n-1}\csc(\frac{i\pi}{n}).
\end{equation}
Assuming $m_0 = 1$ without loss of generality (implying $m = \ep$), we define the normalized Lagrangian $\tilde{L} = \hat{L}/(nm)$. Under cylindrical coordinates
$$
	x=r\cos\theta,\quad y=r\sin\theta,\quad z=z,
$$
this becomes
\begin{equation*}
	\tilde{L}(r,\theta,z,\dot{r},\dot{\theta},\dot{z})=\frac{1}{2}(\dot{r}^{2}+r^{2}\dot{\theta}^{2}+(1+n\ep)\dot{z}^2)+\frac{M(n)\ep}{2r}+\frac{1}{\sqrt{r^{2}+(1+n\ep)^{2}z^{2}}}.
\end{equation*}
After a further coordinate change $\hat{z} = \sqrt{1+n\ep}\,z$, the Legendre transformation gives us 
\begin{equation*}
	p_{r} = \dot{r}, \quad p_{\theta} = r^2\dot{\theta}, \quad p_{\hat{z}} = \dot{\hat{z}},
\end{equation*}
we obtain the reduced Hamiltonian
\begin{equation}
	H_{\om,\ep}(p_{r},p_{\hat{z}},r,\hat{z}) = \frac{1}{2}(p_{r}^{2} + p_{\hat{z}}^{2}) + \frac{\om^2}{2r^2} - \frac{M(n)\ep}{2r} - \frac{1}{\sqrt{r^{2} + (1+n\ep)\hat{z}^{2}}}
\end{equation}
with fixed angular momentum $p_{\theta}=r^2\dot{\theta}=\om$.
Hereafter, we denote $\hat{z}$ simply by $z$ when no confusion arises. The parameter set $\Gamma_\ep$ can be explicitly calculated for this system.
\begin{prop}
	For $-(1+(M(n)\ep)/2)^2<2h\om^2<-(M(n)\ep)^2/4$, energy surface $\mathcal{M}_\ep$ is a compact contact manifold diffeomorphic to $\mathbb{S}^{3}$ with contact form $\lambda$.
\end{prop}
\begin{proof}
	The proof is similar to Proposition \ref{prop;diffeomorphic to S3 1}, leaving it to the reader.
\end{proof}
\begin{thm}
	For sufficiently small $\ep$ and $2h\om^2\in(-\frac{(2+M(n)\ep)^2}{4},-\frac{(M(n)\ep)^2}{4})$ with $2\leq n\leq 472$,
	\begin{enumerate}
		\item[(a)] the n-pyramidal problem has a $\om$-relative periodic solution $\xi_{\ep}(t)$ that lies on the $r$-axis and it is linearly stable. 
		\item[(b)] there exists one and only one $z$-symmetric brake orbit $\xi_{bz,\ep}(t)$ in $\M_\ep$ which forms a Hopf link with $\xi_{\ep}(t)$.
		\item[(c)] the n-pyramidal problem admits infinitely many $\om$-relative periodic solutions in $\M_\ep$.
	\end{enumerate}
\end{thm}
This result clearly implies Theorem \ref{thm:The isosceles three body problem}.
\begin{remark}
	The linear stability of $\xi_{\ep}(t)$ refers to the linear stability in this submanifold. The global linear stability of  $\xi_{\ep}(t)$ can be found in \cite{Moe92,HLO20,HOT23,OS23}.
\end{remark}
\begin{proof}
We interpret this system as a Kepler problem with a symmetric perturbation, where the perturbation term can be written as
\begin{equation*}
	\begin{aligned}
		&f(r,z,\ep)=f(\rho,\varphi,\ep)=-\frac{M(n)}{2\rho\cos\varphi}+\frac{n\sin^{2}\varphi}{2\rho}+O(\ep),\\
		&f(r,0,0)=-\frac{M(n)}{2r}.
	\end{aligned}
\end{equation*}
The existence of $\xi_{\ep}$ follows from Lemma \ref{lem:xi-r-ep}. Then Theorem \ref{thm:z-sym exist} implies property (b). To establish property (c), we apply Theorem \ref{thm:Exist infinity many} through the following analysis.

First, we determine the sign of $D(f)$ by computing
\begin{equation*}
	\tilde{T}(f)=\frac{2}{e}\int_{0}^{\pi} r^{2}(\theta)\cdot\frac{M(n)}{2r^{2}(\theta)}\cos\theta\mathrm{d}\theta=0,\quad E(f)=\left( n-\frac{M(n)}{2}\right)\left(\int_{0}^{\pi}r(\theta)\mathrm{d}\theta\right),
\end{equation*}
we obtain
\begin{equation*}
	\begin{aligned}			
		D(f)&=\left( n-\frac{M(n)}{2}\right)^{2}\left(\int_{0}^{\pi}r(\theta)\mathrm{d}\theta\right)^{2}-\left( n-\frac{M(n)}{2}\right)^{2}\left(\int_{0}^{\pi}r(\theta) \cos2\theta\mathrm{d}\theta\right)^{2}\\
		&=\left( n-\frac{M(n)}{2}\right)^{2}\cdot\frac{\pi^{2}\om^{4}}{1-e^2}\left(1-\left(\frac{-e}{1+\sqrt{1-e^2}}\right)^4\right)\\
		&=\frac{(2n-M(n))^2\pi^{2}\om^{4}G^2(e)}{4(1-e^2)}>0,
	\end{aligned}
\end{equation*}
where $G(e)=\sqrt{1-\left(\frac{-e}{1+\sqrt{1-e^2}}\right)^4}$. Here we used the fact that $\int_{0}^{\pi}\frac{\cos nt}{1+e\cos t}\mathrm{d}t=\frac{\pi}{\sqrt{1-e^2}}\left(\frac{-e}{1+\sqrt{1-e^2}}\right)^n$ for $e\in(0,1)$. By Corollary \ref{coro:linear stable}, $\xi_{\ep}(t)$ is linearly stable and this implies property (a).

When $2\leq n\leq472$, this is equivalent to $2n>M(n)$, therefore $E(f)>0$ and we need verify condition (ii) in Theorem \ref{thm:Exist infinity many}. Then we calculate $\tilde{V}(f)$,
	\begin{equation*}
	 	\tilde{V}(f)=\int_{\mathcal{H}}f(r,z,0)\mathrm{d}r\mathrm{d}z=-\frac{M(n)}{2}\int_{\mathcal{H}}\frac{1}{\cos\varphi}\mathrm{d}\rho\mathrm{d}\varphi+\frac{n}{2}\int_{\mathcal{H}}\sin^2\varphi\mathrm{d}\rho\mathrm{d}\varphi.
	\end{equation*}
For $\int_{\mathcal{H}}\sin^2\varphi\mathrm{d}\rho\mathrm{d}\varphi$, using variable substitution $\sin\varphi=e\sin\theta$ and formula (\ref{formula 1}),
\begin{equation*}
	\begin{aligned}
		&\int_{\mathcal{H}}\sin^2\varphi\mathrm{d}\rho\mathrm{d}\varphi=\frac{1}{-h}\int_{\varphi^-}^{\varphi^+}\sin^2\varphi\left(1+\frac{2h\om^2}{\cos^2\varphi}\right)^{1/2}\mathrm{d}\varphi\\
		=&\frac{1}{-h}\int_{\varphi^-}^{\varphi^+}\tan^2\varphi(e^2-\sin^2\varphi)^{1/2}\mathrm{d}\sin\varphi\\
		=&\frac{1}{-h}\int_{\varphi^-}^{\varphi^+}(\sec^2\varphi-1)(e^2-\sin^2\varphi)^{1/2}\mathrm{d}\sin\varphi\\
		=&\frac{e^2}{-h}\int_{-\pi/2}^{\pi/2}\frac{\cos^2\theta}{\cos^2\theta+(1-e^2)\sin^2\theta}-\cos^2\theta\mathrm{d}\theta\\
		=&\frac{e^2}{-h}\left(\frac{\pi}{1+\sqrt{1-e^2}}-\frac{\pi}{2}\right)=\pi\om^2C^2(e).
	\end{aligned}
\end{equation*}
For $\int_{\mathcal{H}}\frac{1}{\cos\varphi}\mathrm{d}\rho\mathrm{d}\varphi$, using the inequality $(1+\tan^2\varphi)^{3/2}\geq(1+\frac{3}{2}\tan^2\varphi)$,
\begin{equation*}
	\begin{aligned}
		&\int_{\mathcal{H}}\frac{1}{\cos\varphi}\mathrm{d}\rho\mathrm{d}\varphi=\frac{1}{-h}\int_{\varphi^-}^{\varphi^+}\frac{1}{\cos\varphi}\left(1+\frac{2h\om^2}{\cos^2\varphi}\right)^{1/2}\mathrm{d}\varphi\\
		=&\frac{1}{-h}\int_{\varphi^-}^{\varphi^+}(1+\tan^2\varphi)^{3/2}(e^2-\sin^2\varphi)^{1/2}\mathrm{d}\sin\varphi\\
		\geq&\frac{1}{-h}\int_{\varphi^-}^{\varphi^+}(1+\frac{3}{2}\tan^2\varphi)(e^2-\sin^2\varphi)^{1/2}\mathrm{d}\sin\varphi\\
		=&\pi\om^2C(e)\left(\frac{1}{\sqrt{1-e^2}}+1\right)+\frac{3}{2}\pi\om^2C^2(e)\\
		=&\frac{\pi}{2}\om^2C^2(e)+\frac{2\pi\om^2C(e)}{\sqrt{1-e^2}}.
	\end{aligned}
\end{equation*}
Combining above computations, we obtain
\begin{equation*}
	\tilde{V}(f)\leq\frac{\pi}{4}(2n-M(n))\om^2C^2(e)-\frac{M(n)\pi\om^2C(e)}{\sqrt{1-e^2}}.
\end{equation*}
Then we calculate the right-hand side,
	\begin{equation*}
			\tilde{A}(f)=\int_{0}^{\pi}r^{2}(\theta)f(r(\theta),0,0)\mathrm{d}\theta=-\frac{M(n)}{2}\int_{0}^{\pi}r(\theta)\mathrm{d}\theta
			=-\frac{M(n)\pi\om^{2}}{2\sqrt{1-e^{2}}},
	\end{equation*}
and
	\begin{equation*}
			2C(e)\tilde{A}(f)+C^2(e)\sqrt{D(f)}\slash2\\
			=\frac{\pi}{4}(2n-M(n))\om^{2}C^2(e)\frac{G(e)}{\sqrt{1-e^2}}-\frac{M(n)\pi\om^{2}C(e)}{\sqrt{1-e^{2}}}.		
	\end{equation*}
	Due to $\frac{G(e)}{\sqrt{1-e^2}}>1$ for all $e\in(0,1)$, condition (ii) holds.
\end{proof}

\section{Appendix. The Maslov-type index for symplectic paths}

In this section, we briefly review the index theory for symplectic paths. Details can be found in \cite{Lon02}. Let $(\mathbb{R}^{2n},\omega)$ be the standard symplectic vector space with coordinates $(x_1,\cdots,x_n,y_1,\cdots,y_n)$, then $\omega=\sum_{i=1}^{n}\mathrm{d}x_i\wedge\mathrm{d}y_i$. Let $J=\begin{pmatrix}
	0 & -I_n\\
	I_n & 0
\end{pmatrix}$, where $I_n$ is the identity matrix on $\mathbb{R}^n$. The symplectic group $\text{Sp}(2n)$ is defined by
$$\text{Sp}(2n)=\{M\in GL(2n,\mathbb{R})\big|M^TJM=J\},
$$
with topology induced from $\mathbb{R}^{4n^2}$. For $\tau>0$, we define the symplectic paths set
$$\mathcal{P}_\tau(2n)=\{\gamma\in C([0,\tau],\text{Sp}(2n))\big|\gamma(0)=I_{2n}\}.
$$
For any $\omega\in\mathbb{U}$ (the unite circle in $\mathbb{C}$), the following $\omega$-degenerate hypersurface of codimension one in $\text{Sp}(2n)$ is defined by
$$\text{Sp}(2n)_\omega^0=\{M\in\text{Sp}(2n)\big|\det(M-\omega I_{2n})=0\}.
$$
Moreover, the $\omega$-regular set of $\text{Sp}(2n)$ is defined by $\text{Sp}(2n)_\omega^*=\text{Sp}(2n)\setminus\text{Sp}(2n)_\omega^0$. For any two continuous paths $\xi$ and $\eta:[0,\tau]\rightarrow\text{Sp}(2n)$ with $\xi(\tau)=\eta(0)$, we define their concatenation as:
$$\eta*\xi(t)=\begin{cases}
	\xi(2t), & \text{ if } 0\leq t\leq\tau/2;\\
	\eta(2t-\tau), & \text{ if } \tau/2\leq t\leq \tau.
\end{cases}
$$

Given any two $2m_k\times2m_k$ matrices of square block form $M_k=\begin{pmatrix}
	A_k & B_k\\
	C_k & D_k
\end{pmatrix}$ with $k=1,2$, the $\diamond$-product of $M_1$ and $M_2$ defined by the following $2(m_1+m_2)\times2(m_1+m_2)$ matrix $M_1\diamond M_2$:
$$M_1\diamond M_2=\begin{pmatrix}
	A_1 & 0 & B_1 & 0\\
	0 & A_2 & 0 & B_2\\
	C_1 & 0 & D_1 & 0\\
	0 & C_2 & 0 & D_2
\end{pmatrix}.
$$
We define a special continuous symplectic path $\xi_m\subset\text{Sp}(2n)$ by
$$\xi_n(t)=\begin{pmatrix}
	2-\frac{t}{\tau} & 0\\
	0 & \left(2-\frac{t}{\tau}\right)^{-1}
\end{pmatrix}^{\diamond n} \text{ for }  0\leq t\leq\tau.
$$
For $M\in\text{Sp}(2n)_\omega^0$, we define a co-orientation of $\text{Sp}(2n)_\omega^0$ at $M$ by the positive direction $\frac{\mathrm{d}}{\mathrm{d}t}Me^{tJ}\big|_{t=0}$ of the path $Me^{tJ}$ with $|t|$ sufficiently small. Now we will give the definition of $\omega$-index.
\begin{defi}
	For $\omega\in\mathbb{U}$, $\gamma\in\mathcal{P}_\tau(2n)$, \textbf{the $\omega$-index} of $\gamma$ is defined as
	$$i_\omega(\gamma)=\left[e^{-\ep J}\gamma*\xi_n:\text{Sp}(2n)_\omega^0\right],
	$$
	for $\ep>0$ small enough, where $[\cdot:\cdot]$ is the intersection number. The $\omega$-nullity of $\gamma$ is defined as
	$$\nu_\omega(\gamma)=\dim_\mathbb{C}\ker_\mathbb{C}(\gamma(\tau)-\omega I_{2n}).
	$$
\end{defi}
Particularly, if $\gamma(\tau)\in\text{Sp}(2n)_\omega^*$, we can take $\ep=0$.

For any two path $\gamma_j\in\mathcal{P}_\tau(2n_j)$ with $j=1,2$, let $\gamma_1\diamond\gamma_2(t)=\gamma_1(t)\diamond\gamma_2(t)$ for all $t\in[0,\tau]$. We can easily obtain
$$i_\omega(\gamma_1\diamond\gamma_2)=i_\omega(\gamma_1)+i_\omega(\gamma_2),\quad \forall\omega\in\mathbb{U}.
$$
For any symplectic path $\gamma\in\mathcal{P}_\tau(2n)$ and $m\in\mathbb{Z}^+$, we define the $m$-th iteration $\gamma^m:[0,m\tau]\rightarrow\text{Sp}(2n)$ by
$$\gamma^m(t)=\gamma(t-j\tau)\gamma(\tau)^j,\quad\text{ for } j\tau\leq t\leq(j+1)\tau,j=0,1,\cdots m-1.
$$
\begin{defi}
	The \textbf{mean index} of $\gamma$ is defined as
	$$\hat{i}(\gamma)=\lim\limits_{m\rightarrow+\infty}\frac{i_1(\gamma^m)}{m}.
	$$
\end{defi}
Let $x(t)$ be a periodic solution of Hamiltonian system $\dot{x}=J\nabla H(x)$ with period $\tau$, then the fundamental solution of the linearized equation $$\dot{y}=JD^2H(x(t))y$$ is a symplectic path starting form $I_{2n}$, denote it by $\gamma(t)$. Then we can define the $\omega$-index and mean index for $x(t)$.
\begin{defi}
	The $\omega$-index and $\omega$-nullity of $x(t)$ is defined as
	$$i_\omega(x)=i_\omega(\gamma),\quad \nu_\omega(x)=\nu_\omega(\gamma).
	$$
	The mean index of $x(t)$ is defined as
	$$\hat{i}(x)=\hat{i}(\gamma).
	$$
\end{defi}
For mechanical systems with $\gamma\in\mathcal{P}_\tau(2)$, the analysis is simpler than the general case. Let $\sigma(t)=\rho(t)e^{i\vartheta(t)}$ and $\sigma^{-1}(t)$ denote the eigenvalues of $\gamma(t)$, where $\vartheta(t)$ is a continuous non-decreasing function. The mean index of $\gamma$ is given by  
$$
\hat{i}(\gamma) = \frac{\vartheta(\tau) - \vartheta(0)}{\pi}.  
$$

\hfill\newline
\noindent{\bf Data Availability.} Data sharing not applicable to this article as no datasets were generated or analyzed during the current study.

\hfill\newline
\noindent{\bf Acknowledgement.} 
The authors are deeply grateful to the anonymous referees for their careful reading of the manuscript and insightful comments and suggestions.

\hfill\newline
\bibliographystyle{abbrv}
\bibliography{RefKepler1}

\end{document}